\documentclass[12pt]{article}
\usepackage{amsmath}
\usepackage{amssymb}
\usepackage{amsthm}
\usepackage{citehack}
\newtheorem{condition}{ Condition}[section]
\newtheorem{definition}{ Definition}[section]
\newtheorem{theorem}{Theorem}[section]
\newtheorem{lemma}{Lemma}[section]
\newtheorem{proposition}{Proposition}[section]
\newtheorem{corollary}{Corollary}[section]

\newtheorem{remark}{{\it Remark}}[section]

\renewcommand{\div}{\text{\rm div}\,}

\newcommand{\vo}{\text{\bf u}\otimes}

\def\div{\mathop{\mathrm{div}}}

\makeindex

\title{Isothermal Navier-Stokes Equations and  Radon Transform}

\author{ P. I. Plotnikov,  W. Weigant,\\
{\small Lavryentyev Institute of Hydrodynamics,}\\{\small Lavrentyev
pr.~15, Novosibirsk 630090,
Russia (plotnikov$@$hydro.nsc.ru).}\\
{ \small Universit\"{a}t Bonn, Institute f\"{u}r Angewandte
Mathematik,}\\{\small Endenicher Alle 60, 53115 Bonn, Germany .} }
\date{}
\begin{document}

\setcounter{page}{1} \maketitle

\begin{abstract}
In the  paper we prove  the existence results for initial-value
 boundary value problems for
 compressible isothermal Navier-Stokes equations. We restrict ourselves to
2D case of a problem with no-slip condition for nonstationary motion
of viscous compressible  isothermal  fluid. However, the technique
of modeling and analysis presented here is general and can be used
for 3D problems.
\end{abstract}

AMS:35Q30, 49J20, 76N10

\pagestyle{myheadings} \thispagestyle{plain} \markboth{P. I.
PLOTNIKOV AND W. WEIGANT }{ISOTHERMAL NAVIER-STOKES EQUATIONS}

\section{Introduction}
\subsection{Problem formulation}
Suppose a viscous compressible fluid occupies
  a bounded domain $\Omega\subset \mathbb{R}^2$.
The  state of the fluid is characterized  by the macroscopic
quantities: the {density}
 $\varrho(x,t)$ and the {velocity}
 $\mathbf{u}(x,t)$.
  The problem is to  find
$\mathbf u(x,t)$ and $\varrho(x,t)$ satisfying the following
equations and boundary conditions in the cylinder $Q_T=\Omega\times
(0,T)$.
 \begin{subequations}
  \label{stokes1}
\begin{gather}
\label{stokes2} \partial_t (\varrho \mathbf{u})+\div(\varrho
\vo\mathbf{u})+\nabla
\varrho=\div\mathbb S(\mathbf u)+\varrho\mathbf f\quad\text{in~~}Q_T,\\
 \label{stokes3}\partial_t\varrho+\div(\varrho \mathbf{u})=0\quad\text{in~~}Q_T,\\\label{stokes4}
\mathbf u=0\quad\text{on~} \partial\Omega\times (0,T),\\
\label{stokes5} \mathbf u(x,0)=\mathbf u_0(x),\quad
\varrho(x,0)=\varrho_0(x)\quad\text{in~} \Omega.
\end{gather}
Here, the vector field $\mathbf f$  denotes the density of external
mass forces,
 the {viscous stress
tensor}
 $\mathbb S(\mathbf u)$ has the form
\begin{equation}
\label{stokes4a}
 \mathbb S(\mathbf{u})=\nu_1 \big(\nabla \mathbf{u}+\nabla\mathbf{u}^\top\big)+ \nu_2 \text{div~}
 \mathbf u\,\mathbb{I},
\end{equation}
in which the {viscosity coefficients} satisfy the inequalities
$\nu_1>0$, $\nu_1+\nu_2\geq 0$.
\end{subequations}
It is necessary to notice that problem \eqref{stokes1} is the
simplest multidimensional boundary value problem for the
compressible Navier-Stokes equations. In  1986  Padula, see
\cite{padula1}, formulated the  result on existence of a weak
solution to problem \eqref{stokes1}, but the proof presented was
incomplete, see \cite{padula2}. The first nonlocal  results
concerning the mathematical theory of compressible Navier-Stokes
equations are due to P.-L. Lions. In monograph \cite{PLL} he
established the existence of a  renormalized solution to
nonstationary boundary value problem for the Navier-Stokes equations
with the pressure function $p\sim\varrho^\gamma$ for all
$\gamma>5/3$ in $3D$ case and for all $\gamma>3/2$ in $2D$ case.
 More recently, Feireisl, Novotn\'{y},  and Petzeltova', see
  \cite{FNP},  proved the existence result for all
 $\gamma>3/2$ in $3D$ case and for all $\gamma>1$ in $2D$ case,
 see also monographs \cite{FEIRBOOK}, \cite{NovotnyStra}, and \cite{PlotnSoc}
 for references and details.
The question on solvability of problem \eqref{stokes1} remained
open. The main difficulty is the so called concentration problem,
see \cite{PLL} ch.6.6. This means that the finite kinetic energy can
be concentrated in very small domains. Our goal is to relax the
restriction $\gamma>1$ and to prove the  existence of solutions to
problem \eqref{stokes1}. In order to make the presentation clearer
and avoid unnecessary technical difficulties, we assume that the
flow domain  and the given data satisfy  the  hypotheses:
\begin{condition}
\label{stokes6}
\begin{itemize}
\item
The flow domain  $\Omega\subset \mathbb{R}^2$ is a bounded domain
with $C^\infty$ boundary.
\item The  data satisfy
$ \varrho_0, \mathbf{u}_0\in L^\infty(\Omega)$, $\mathbf f\in
L^\infty(Q_T)$, and
\begin{equation}\label{stokes8}
\|\mathbf{u}_0\|_{W^{1,2}_0(\Omega)}+
\|\varrho_0\|_{L^\infty(\Omega)} +\|\mathbf f\|_{L^\infty(Q_T)}\leq
c_e,\quad \varrho_0>c>0,
\end{equation}
where $c_e$, c are  positive constants.
\end{itemize}
\end{condition}

\begin{remark}\label{stokesremark}
Further, we denote by $E$  generic constants depending only on $
\Omega, T, \|\varrho_0\|_{L^\infty(\Omega)}$,   $\|\mathbf
u_0\|_{L^2(\Omega)}, $  $ \|\mathbf f\|_{L^\infty(Q_T)}$, and
$\nu_i$.
\end{remark}

We claim that problem  \eqref{stokes1} admits a weak solution which
is defined as follows:
\begin{definition}\label{stokes9} A couple
$$\varrho\in L^\infty(0,T; L^1(\Omega)), \quad \mathbf{u}\in L^2(0,T; W^{1,2}_0(\Omega))
$$
is said to be a weak  solution
  to problem
\eqref{stokes1} if  $(\varrho, \mathbf u)$ satisfies
\begin{itemize}
\item The kinetic energy  is  bounded, i.e., $\varrho|\mathbf{u}|^2\in L^\infty(0, T; L^1(\Omega))$.
The density function is non-negative $\varrho\geq 0$.
\item The integral identity
\begin{multline}\label{stokes10}
\int_{Q_T}\big (\varrho\mathbf{u}\cdot
\partial_t\boldsymbol{\xi}+
\varrho\mathbf{u}\otimes\mathbf{u}:\nabla\boldsymbol{\xi}+\varrho\text{\rm
div~} \boldsymbol{\xi}-\mathbb
S(\mathbf{u}):\nabla\boldsymbol{\xi}\big)\,dxdt \\+\int_{Q_T}
\varrho\mathbf f\cdot\boldsymbol{\xi}\, dxdt+ \int_{\Omega
}(\varrho_0 \mathbf u_0)(x)\cdot \boldsymbol{\xi})(x,0)\,dx=0
\end{multline}
holds  for all vector fields $\boldsymbol{\xi}\in C^\infty(Q_T)$
vanishing  in a neighborhood of $\partial\Omega\times[0,T]$ and of~
$\Omega\times\{t=T\}$.
\item The integral identity
\begin{equation}\label{stokes11}
\int_{Q_T}\big (\varrho
\partial_t\psi+\varrho\mathbf{u}\cdot\nabla\psi
\big)\,dxdt+\int_{\Omega } \varrho_0(x)\psi(x,0)\, dx=0
\end{equation}
holds  for all  $\psi\in C^\infty(Q_T)$ vanishing in a neighborhood
of the top~ $\Omega\times\{t=T\}$.
\end{itemize}
\end{definition}
The following existence theorem  is the main result of the paper.
\begin{theorem}\label{stokes14}
Assume  that Condition  \ref{stokes6} is fulfilled. Then
 problem  \eqref{stokes1}  has a
weak  solution which meets all requirements of Definition
\ref{stokes9} and satisfies the estimate
\begin{gather}\label{stokes12}
\|\mathbf u\|_{L^2(0,T; W^{1,2}_0(\Omega))} +\|\varrho |\mathbf
u|^2\|_{L^\infty(0,T;L^1(\Omega))}+
\|\varrho\log(1+\varrho)\|_{L^\infty(0,T; L^1(\Omega))}\leq E,
\end{gather}
where the constant $E$ is as in Remark \ref{stokesremark}.
\end{theorem}

The next theorem, which is the second main result of the paper,
shows that a weak solution to problem \eqref{stokes1}  has extra
regularity properties.
 \begin{theorem}\label{stokes16} Let   Condition  \ref{stokes6} be satisfied. Assume that
 $( \varrho, \mathbf u)$  meets all requirements
 of Theorem \ref{stokes14}. Furthermore assume that $\mathbf u$ and $\varrho$ are extended by $0$ to $\mathbb R^2\times (0,T)$.  Then for every nonnegative function $\zeta\in C^\infty_0(\mathbb R^2)$ with
 $\text{\rm spt~}\zeta\Subset\Omega$,
 \begin{equation}\label{stokes17}
  \text{\rm ess}\sup\limits_{\boldsymbol{\omega}\in \mathbb S^1}  \int\limits_0^T\int\limits_{-\infty}^\infty \Phi(\boldsymbol{\omega},\tau,t)^2\, d\tau dt\leq c(\zeta)E,
 \end{equation}
 where $\Phi$ is the Radon transform of $\zeta(x)\varrho(x,t)$,
 \begin{equation}\label{stokes18}
  \Phi(\boldsymbol{\omega}, \tau, t) =\int\limits_{  \boldsymbol{\omega}\cdot x=\tau}
  \zeta(x)\varrho(x,t)\, dl.
 \end{equation}
 Moreover, the function $\zeta\varrho$ admits the estimates
 \begin{equation}\label{19}\begin{split}
    \|\zeta\varrho\|_{L^2(0,T; H^{-1/2}(\mathbb R^2))}\leq c(\zeta)E, \\
    \|\zeta\varrho\|_{L^{1+\lambda}(Q_T)}\leq c(\zeta, \lambda)E \text{~~for all~~}
    \lambda\in [0, 1/6).
\end{split} \end{equation}
 Here $c(\zeta)$ depends only on $\zeta$ and $c(\zeta,\lambda)$ depends only on $\zeta$, $\lambda$.
 \end{theorem}

 The remaining part of the paper is devoted to the proof of these theorems.
In sections \ref{preliminaries} and \ref{regularized} we collect
basic facts on Sobolev spaces, the Radon transform, and the
isentropic  Navier-Stokes equations. Section \ref{radon} is the
heart of the work. Here  we derive the $L^2$-estimates for the Radon
transform of the density function $\varrho$. In sections
\ref{momentum} and \ref{lpestimate} we  prove that the density is
locally  integrable with exponent $1+\lambda<7/6$. In section
\ref{final} we complete the proof of Theorems \ref{stokes14} and
\ref{stokes16}.

\section{Preliminaries}\label{preliminaries}
\subsection{Sobolev spaces. Radon transform. Multiplicators}
For every $s\in \mathbb R$, denote by $H^s(\mathbb R^2)$ the Sobolev
space of all tempered distributions $u$ in $\mathbb R^2$ with the
finite norm
\begin{equation}\label{sobolev1}
    \|u\|_{H^s(\mathbb R^2)}=\|(1+|\xi|^2)^{s/2}\,\mathfrak F{u}\|_{L^2(\mathbb R^2)},
\end{equation}
where $\mathfrak F{u}(\xi)$ is the Fourier transform of $u$. For all
nonnegative integers $k$, the space $H^k(\mathbb R^2)$ coincides
with $W^{k,2}(\mathbb R^2)$.   For every $u\in L^2(\mathbb R^2)$ and
$s\geq 0$ we have
\begin{equation}\label{sobolev2}
    \|u\|_{H^{-s}(\mathbb R^2)}=\sup\limits_{g\in H^{s}(\mathbb R^2)} \frac{\int_{\mathbb R^2}ug\,dx}{
    \|g\|_{H^{s}(\mathbb R^2)}}.
\end{equation}
Introduce  the  Bessel kernel $G_1=\mathfrak
F^{-1}(1+|\xi|^2)^{-1/2}$.  It is well-known that it  is strictly
positive and  analytic in $\mathbb R^2\setminus\{0\}$. Moreover, the
Bessel kernel admits the estimates
\begin{equation}\label{sobolev3}\begin{split}
  c^{-1} |z|^{-1}\leq G_1(z)\leq  c |z|^{-1}\text{~~for~~} |z|\leq 1, \quad
   G_1(z)\leq  c |z|^{-1}e^{-|z|}\text{~~for~~} |z|\geq 1.
\end{split}\end{equation}
In particular, for every $N>0$ there exists a constant $e(N)>0$ with
the property
\begin{equation}\label{sobolev4}\begin{split}
  e(N) |z|^{-1}\leq G_1(z)\leq  c |z|^{-1}\text{~~for~~} |z|\leq N.
\end{split}\end{equation}
The equality
\begin{equation}\label{sobolev5}
    \|G_1* u\|_{H^{s+1}(\mathbb R^2)}\,=\|u\|_{H^{s}(\mathbb R^2)}.
\end{equation}
holds true for all $u\in H^s(\mathbb R^2)$, $s\in\mathbb R$.

The next lemma constitutes  Sobolev estimates for the functions with
integrable Radon transform.
\begin{lemma}\label{sobolevlemma1} Let $g\in L^2(\mathbb R^2)$ be a compactly supported.
Then
\begin{equation}\label{sobolev6}
    \|g\|_{H^{-1/2}(\mathbb R^2)}^2\leq
    \frac{1}{4\pi}\int\limits_{\mathbb S^1\times \mathbb R} \Phi(\boldsymbol \omega, \tau)^2\,d\boldsymbol \omega\, d\tau, \text{~~where~~}\Phi(\boldsymbol \omega, \tau)=\int\limits_{\boldsymbol \omega \cdot x=\tau} g(x)\, dl.
\end{equation}
\end{lemma}
\begin{proof} The proof is in Appendix \ref{applem1}
\end{proof}

 The last lemma  concerns   multiplicative
 properties of   Sobolev spaces.
\begin{lemma}\label{sobolevlemma2} Let $s>1/2$, $g\in L^2(\mathbb R^2)$  and $u\in H^1(\mathbb R^2)$.
Then there is $c(s)>0$ such that
\begin{equation}\label{sobolev8}
    \|gu\|_{H^{-s}(\mathbb R^2)}\leq c(s) \|g\|_{H^{-1/2}(\mathbb R^2)}\|u\|_{H^{1}(\mathbb R^2)}.
\end{equation}
\end{lemma}
\begin{proof} The proof is in Appendix \ref{applem1}
\end{proof}

\subsection{Poisson  equation}
Let $\Omega\subset \mathbb R^2$ be a bounded domain and $r\in
(1,\infty)$. Let   $f\in L^r(\mathbb R^2)$ be an arbitrary function
 such that $\text{spt~}f\subset\Omega$. Then, see \cite{evans},
 the Poisson equation
\begin{equation}\label{elliptic.1}
\Delta u=f\quad\text{in~} \mathbb{R}^2,
\end{equation}
 has a solution with the properties:
This solution is analytic outside of $\Omega$, and satisfies
\begin{gather*}
\limsup_{|x|\to\infty}\,(\log|x|)^{-1}|u(x)|<\infty,\quad
\|u\|_{W^{1,2}(B_R)}\leq c \|f\|_{L^{r}(\mathbb{R}^d)}.
\end{gather*}
Here  $B_R$ is the ball $\{ x\in\mathbb{R}^d: |x|<R\}$ of an
arbitrary radius $R<\infty$, and the constant $c$ depends only on
$R$ and $\Omega$. The relation $f\to u$ determines a linear operator
$\Delta^{-1}$. In this framework we can define  the linear operators
$$
 A_j=\partial_{x_j}\Delta^{-1}, \quad R_{j}=
\partial_{x_j}(-\Delta)^{-1/2}, \quad
j=1,2.
$$
The Riesz operator
 $R_{j}$ is a singular integral operator and
 by  the Zygmund-Calder\'on theorem
 it is bounded in any space $L^p(\mathbb R^d)$ with $1<p<\infty$.
 In particular we have
\begin{align*}
\|A_jf\|_{W^{1,r}(B_R)}&\leq
c(R,\Omega)\|f\|_{L^{r}(\mathbb{R}^2)}\text{~when~~}\text{spt}
f\subset \Omega,
\\
\|R_{j}f\|_{L^{p}(\mathbb{R}^2)}&\leq c(r) \|f\|_{L^{r}(\mathbb
R^2)}.
\end{align*}
Notice that these operators have integral representations. In
particular, we have
\begin{equation}\label{riesz1}
A_i f(x)=c\int_{\mathbb R^2}|x-y|^{-2} (x_i-y_i)f(y)\, dy.
\end{equation}

\section{Regularized problem}\label{regularized}

In order to regularize problem \eqref{stokes1} we use the artificial
pressure method and replace equations \eqref{stokes1} by regularized
equations
\begin{subequations}
  \label{navier1}
\begin{gather}
\label{navier2} \partial_t (\varrho \mathbf{u})+\div(\varrho
\vo\mathbf{u})+\nabla
p(\varrho)=\div\mathbb S(\mathbf u)+\varrho\mathbf f\quad\text{in~~}Q_T,\\
 \label{navier3}\partial_t\varrho+\div(\varrho \mathbf{u})=0\quad\text{in~~}Q_T,\\\label{navier4}
\mathbf u=0\quad\text{on~} \partial\Omega\times (0,T),\\
\label{navier5} \mathbf u(x,0)=\mathbf u_0(x),\quad
\varrho(x,0)=\varrho_0(x)\quad\text{in~} \Omega.
\end{gather}
Here, the artificial pressure function  is given by
\begin{equation}\label{navier6}
    p(\varrho)=\varrho+\varepsilon \varrho^\gamma, \quad \varepsilon\in (0,1], \quad \gamma\geq 6.
\end{equation}
\end{subequations}
The existence of weak renormalized solutions to problem
\eqref{navier1} was established in monographs \cite{FEIRBOOK} and
\cite{PLL}. The following proposition is a consequence of these
results.
\begin{proposition}\label{navierproposition}
Let domain $\Omega$, and functions $\mathbf u_0$, $\varrho_0$,
$\mathbf f$ satisfy Condition \ref{stokes6}. Then
 problem  \eqref{navier1}  has a
weak  solution $( \varrho, \mathbf u)$ with the following
properties:
\begin{itemize}

\item[{ (i)}]
The   functions $\varrho\geq 0$ and $\mathbf u$ satisfy the energy
inequality
\begin{equation}\label{navier7}
\mathop{\rm ess\,sup}_{t\in(0,T)} \int_{\Omega}
\big\{\varrho|\mathbf u|^2+ \varrho \ln (1+\varrho)+\varepsilon
\varrho^\gamma\big\}(x,t)\, dx+ \int_{Q_T}|\nabla\mathbf{u}|^2
\,dxdt\leq c E.
\end{equation}
The constant $E$ is  as in Remark \ref{stokesremark}.
\item[{ (ii)}]The integral identity
\begin{multline}
\int_{Q_T}\big(\varrho \mathbf{u}
\cdot\partial_t\boldsymbol{\xi}\big)\, dxdt + \int_{Q_T}\big(\varrho
\mathbf u\otimes\mathbf u+p(\varrho)\,\mathbb I- \mathbb
S(\mathbf{u})
 \big):\nabla
\boldsymbol{\xi}\, dxdt+\\\int_{Q_T}\varrho\mathbf{f}\cdot
\boldsymbol{\xi}\, dxdt \label{navier8} +\int_\Omega\varrho_0(x)
\mathbf{u}_0(x)\cdot\boldsymbol{\xi}(x,0)\, dx=0
\end{multline}
holds  for all vector fields  $\boldsymbol{\xi}\in C^\infty(Q)$
satisfying
\begin{equation}\begin{split}\label{navier9}
 \boldsymbol{\xi}(x,T)=0\quad\text{in~} \Omega,\quad
 \boldsymbol{\xi}(x,t)=0\quad\text{on~} \partial\Omega\times (0,T).
\end{split} \end{equation}
\item[{ (iii)}]
The integral identity
\begin{multline}\label{navier10}
\int_{Q_T}\Big(\varphi(\varrho)\partial_t\psi+
\big(\varphi(\varrho)\mathbf u\big)
\cdot\nabla\psi-\psi\big(\varphi'(\varrho)\varrho-\varphi(\varrho)\big)\div\mathbf
u\Big)\,dxdt\\
+\int_{\Omega }( \psi\varphi(\varrho_0))(x,0)\, dx= 0
\end{multline}
holds for  all smooth  functions $\psi$, vanishing  in a
neighborhood of the top $\Omega\times \{t=T\}$, and for all
functions $\varphi\in  C^2[0,\infty)$ satisfying the growth
condition
\begin{equation}\label{navier11}
|\varphi(\varrho)|+|\varphi'(\varrho)\varrho|+|\varphi''(\varrho)\varrho^2|\leq
C(1+\varrho^2).
\end{equation}
\end{itemize}
\end{proposition}
\begin{remark}\label{navier17}
 Further we will  assume that
$( \varrho, \mathbf u)$ and $( \varrho_0, \mathbf u_0)$ are extended
by $0$ to  the layer
\begin{equation}\label{glok}
    \Pi\,=\,\mathbb R^2\times (0,T).
\end{equation}
\end{remark}
The following consequences of Proposition \ref{navierproposition}
will be used throughout the paper.
\begin{corollary}\label{navier12}  Assume that $(\varrho, \mathbf u)$ meets all requirements
of Proposition \ref{navierproposition}. Then there is a constant
$c(E, \varepsilon)$, depending only on $E$, $\beta$, $\alpha$,
$\gamma$, and $\varepsilon$, such that
\begin{align*}
&\|\varrho \mathbf{u}\|_{L^\infty(0,T;
L^{2\gamma/(\gamma+1)}(\Omega))}\leq c(E, \varepsilon),
\\
&\|\varrho \mathbf{u}\|_{L^2(0,T; L^\beta(\Omega))}\leq c(E,
\varepsilon)\quad\text{for all~}\, \beta\in [1,
\infty]\quad\text{with}\quad \beta<\gamma,
\\
&\|\varrho \mathbf{u}\|_{L^\alpha(0,T; L^2(\Omega))}\leq c({E},
\varepsilon) \quad\text{for all~}\, \alpha\in[ 1, 2\gamma-2).
\end{align*}
\end{corollary}
\begin{proof}
It follows from \eqref{navier7} that
\begin{equation}\label{einequality}\begin{split}
\|\varrho\|_{L^\infty(0,T; L^\gamma(\Omega))}+
\|\varrho|\mathbf{u}|^2\|_{L^\infty(0,T;L^1(\Omega))} \leq c(E, \varepsilon),\\
\|\mathbf{u}\|_{L^2(0,T; W^{1,2}_0(\Omega))}\leq c(E, \varepsilon).
\end{split}\end{equation}
Hence $( \varrho, \mathbf u)$ are bounded energy functions and the
corollary is a particular case of Corollary 4.2.2  in
\cite{PlotnSoc}.
\end{proof}
\begin{corollary}\label{navier13}
Assume that $(\varrho, \mathbf u)$ meets all requirements of
Proposition \ref{navierproposition}. Then there is a constant $c(E,
\varepsilon)$, depending only on $E$,  $\gamma$, and $\varepsilon$,
such that
\begin{align}\label{navier14}
\|\varrho|\mathbf{u}|^2\|_{L^2(0,T; L^\tau(\Omega))}&\leq
c(E,\epsilon)\quad\text{for all~}\, \tau\in \big[1,
2\gamma/(\gamma+1)\big),
\\
\label{navier15} \|\varrho|\mathbf{u}|^2\|_{L^1(0,T;
L^\tau(\Omega))}&\leq c(E, \varepsilon)\quad\text{for all~}\,
\tau\in \big[1, \gamma\big).
\end{align}
\end{corollary}
\begin{proof}
In view of \eqref{einequality} the corollary is a particular case of
Corollary  4.2.3 in \cite{PlotnSoc}.
\end{proof}

\section{Radon transform}\label{radon}
In this section we estimate the Radon transform of  solutions to
regularized equations \eqref{navier1}. The corresponding result is
given by the following theorem. Fix an arbitrary function $\zeta$
with the properties
\begin{equation}\label{radon1}
    \zeta\in C^\infty_0(\mathbb R^2), \quad \text{spt~} \zeta \Subset \Omega, \quad \zeta\geq 0.
\end{equation}
\begin{theorem}\label{radontheorem}
Assume that a renormalized solution to problem \eqref{navier1} meets
all requirements of Proposition \ref{navierproposition}.
Furthermore, assume that $\mathbf u$ and $\varrho$ are extended by
$0$ to the layer $\Pi$. Then for every unit vector
$\boldsymbol{\omega}\in \mathbb R^2$,
\begin{equation}\label{radon2}
    \int_0^T\int_{-\infty}^\infty \Big\{\int_{\boldsymbol{\omega}\cdot x=\tau}\zeta(x)\varrho(x,t)
    \, dl\Big\}^2\, d\tau dt\leq c(\zeta) E,
\end{equation}
where $c(\zeta)$ depends only on $\zeta$, and $E$ is specified by
Remark \ref{stokesremark}. Notice that $c(\zeta)$ and $E$ are
independent of $\boldsymbol\omega$ and $\varepsilon$.
\end{theorem}

Since the Navier- Stokes equations are invariant with respect to
rotations, it suffices to prove  \eqref{radon2} for $\boldsymbol
\omega=(1,0)$, i.e.,
 to prove the inequality
\begin{equation}\label{radon3}
    \int_0^T\int_{\mathbb R} \Big\{\int_{\mathbb R}\zeta(x)\varrho(x,t)\, dx_2
\Big\}^2\, dx_1 dt\leq c(\zeta) E.
\end{equation}
We split the proof of \eqref{radon3} into a sequence of lemmas.
\begin{lemma}\label{radonlemma1}
Let all hypotheses of Theorem \ref{radontheorem} be satisfied. Then
for every function $\varphi\in C^\infty(Q_T)$ vanishing in a
neighborhood of $\partial\Omega\times (0,T)$,
\begin{multline}
\int_{Q_T}\Big(\varrho u_1 \cdot\partial_t\varphi + \big(\,\varrho
u_1u_i - \mathbb S(\mathbf{u})_{i1}
 \,\big)\frac{\partial\varphi}{\partial x_i}+p\frac{\partial\varphi}{\partial x_1}
 \Big)\, dxdt+\\\int_{Q_T}\varrho f_1\,\varphi
\, dxdt \label{radon4} \leq c(\zeta)E\|\varphi\|_{L^\infty(Q_T)}.
\end{multline}
\end{lemma}
\begin{proof} Set
\begin{equation}\label{etah}
\eta_h(t)=1\text{~~for~~} t\leq T-h, \quad
\eta=\frac{1}{h}(T-t)\text{~~for~~} t\in[T-h,T].
\end{equation}
Substituting $\boldsymbol \xi=(\eta_h\varphi,0)$ into
\eqref{navier8} we arrive at the identity
\begin{multline}
\int_{Q_{T}}\eta_h\varrho u_1 \cdot\partial_t\varphi \,dxdt+
\int_{Q_{T}}\eta_h\Big(\big(\,\varrho u_1u_i - \mathbb
S(\mathbf{u})_{i1}
 \,\big)\frac{\partial\varphi}{\partial x_i}+
 p\frac{\partial\varphi}{\partial x_1}\Big)\, dxdt+\\\int_{Q_T}\eta_h\varrho f_1
 \varphi\,dxdt=
 \frac{1}{h}\int_{T-h}^{T}\int_\Omega \varrho \mathbf u_1\varphi \,dxdt-
 \int_\Omega \varrho_0\mathbf u_{1,0} \varphi(x,0)\,dx.
\label{radon5}
\end{multline}
Next notice that
\begin{multline*}
\Big|\frac{1}{h}\int_{T-h}^{T}\int_\Omega \varrho \mathbf u_1\varphi
\,dxdt\Big|+
 \Big|\int_\Omega \varrho_0\mathbf u_{1,0} \varphi(x,0)\,dx\Big|\leq\\
 \|\varrho \mathbf u\|_{L^\infty(0,T; L^1(\Omega))}\|\varphi\|_{L^\infty(Q_T)}+
 \|\varrho_0 \mathbf u_0\|_{ L^1(\Omega)} \|\varphi\|_{L^\infty(Q_T)}\leq E \|\varphi\|_{L^\infty(Q_T)}.
\end{multline*}
Letting $h\to 0$  in \eqref{radon5} we arrive at \eqref{radon4}
\end{proof}

Now we specify the test function    $\varphi$. Choose
$\omega:\mathbb R\to \mathbb R^+$ satisfying
$$
\omega\in C^\infty_0(\mathbb R), \quad \text{spt~}\omega \subset
[-1,1], \quad \omega \text{~~is even~}, \quad \int_{\mathbb R}
\omega(s)\, ds=1,
$$
For every  $f\in L^1_{\rm loc}(\mathbb R^2)$ define the mollifiers
\begin{equation}\label{radon7}
    \Big[ f\Big]_h=\frac{1}{h^2}
    \int_{\mathbb R^2}\omega\big(\frac{x_1-y_1}{h}\big)
    \omega\big(\frac{x_2-y_2}{h}\big)f(y)\, dy.
\end{equation}
Introduce the auxiliary functions
\begin{subequations}\label{radon9}
\begin{equation}\label{radon9a}
H(x_1,t)=\int_{-\infty}^{x_1} \Psi(s,t)\, ds,\quad
  \Psi(x_1,t)=\int_{\mathbb R } \big[\zeta \varrho\big]_h (x_1, x_2, t)\,
  dx_2,
  \end{equation}
and take the test function $\varphi$ in the form
\begin{equation}\label{radon9b}
 \varphi(x,t)=\zeta(x) \, \big[\, H\,\big]_h(x_1,t).
\end{equation}
\end{subequations}
The following lemma constitutes properties of $\Psi$ and $H$.
\begin{lemma}\label{radonlemma2}
 $\Psi, H\in L^\infty(0,T; C^k(\mathbb R))$ for every integer $k\geq0$, and
\begin{equation}\label{radon10}
    \|H\|_{L^\infty(\mathbb R\times (0,T))}\leq c(\zeta)E.
\end{equation}
\end{lemma}
\begin{proof}
 Notice that $\zeta\varrho\in L^\infty(0,T; L^1(\mathbb R^2))$ and its norm in this space does not exceed $E$. Hence for a.e. $t\in(0,T)$,
$$
\|\big[\zeta\varrho\big]_h(t)\|_{C^k(\mathbb R^2)}\leq c(k)
\|\zeta\varrho(t)\|_{L^1(\mathbb R^2)}\leq c(k)E.
$$
Hence $\big[\zeta\varrho\big]_h$ belongs to $L^\infty(0,T;
C^k(\mathbb R))$. Next, there is $N$ such that the square  $[-N+1,
N-1]^2$ contains domain $\Omega$. Hence the function
$\big[\zeta\varrho\big]_h(t)$ is compactly supported in the square
$[-N,N]^2$. It follows that $\Psi\in L^\infty(0,T; C^k(\mathbb R))$
and $\Psi(\cdot, t)$ is supported in the interval $[-N,N]$. From
this and \eqref{radon9} we conclude that $H\in L^\infty(0,T;
C^k(\mathbb R))$. It remains to note that
$$
|H(x_1,t)|\leq \int_{\mathbb R}\Psi(x_1,t)dx_1\leq \int_{\mathbb
R^2}\big[\zeta\varrho\big]_hdxdt =\int_{\mathbb
R^2}\zeta\varrho\,dxdt=\int\limits_\Omega \zeta(x)\varrho(x,t)dx\leq
cE.
$$
\end{proof}

Now we investigate in details the time dependence of $H$.

\begin{lemma}\label{radonlemma3} The function $\partial_tH$ belongs to the
class $L^\infty(0,T; C^k(\Omega))$ for every integer $k\geq 0$.
Moreover, it has the representation
\begin{equation}\label{radon11}
    \partial_t H=-\upsilon_h +J_0,\text{~~where~~} \upsilon_h(x_1, t)=\int_{\mathbb R}\big[\zeta\varrho u_1\big]_h(x_1,x_2, t)\, dx_2,
\end{equation}
and the reminder admits the estimate
\begin{equation}\label{radon13}
    |J_0|\leq c(\zeta) E.
\end{equation}
\end{lemma}
\begin{proof} Integral identity
\eqref{navier10} with
 $\varphi(\varrho)=\varrho$ and
  $\psi$ replaced  by $\zeta\psi$ reads
\begin{equation}\label{radon14}
\int_{\Pi}\Big(\zeta\varrho\,\partial_t\psi+ \zeta \varrho\mathbf u
\cdot\nabla\psi +\psi \varrho \nabla\zeta\mathbf u\Big)\,dxdt
+\int_{\mathbb R^2 } \psi(x,0)\zeta(x)\varrho_0(x)\, dx= 0.
\end{equation}
This identity holds true for  all  functions $\psi\in
C^\infty(\mathbb R^2\times (0,T)$, vanishing  in a neighborhood of
the top $\mathbb R^2\times \{t=T\}$. Now choose an arbitrary $\xi\in
C^\infty_0(0,T)$ and $y\in \mathbb R^2$. Inserting
$$
\psi=\xi (t)\,h^{-2}\omega\big(\frac{x_1-y_1}{h}\big)
    \omega\big(\frac{x_2-y_2}{h}\big)
$$
into \eqref{radon14} we arrive at
$$
\int_0^T\Big(\big[\zeta\varrho\big]_h(y,t)\xi'(t)-\xi\div\big[\zeta\varrho\mathbf
u\big]_h(y,t) +\xi\big[\varrho\nabla\zeta\big]_h(y,t)\Big)dt=0,
$$
which yields
\begin{equation}\label{radon15}
   \partial_t  \big[\zeta\varrho\big]_h=-\div \big[\zeta\varrho\mathbf u\big]_h+\big[\varrho\nabla\zeta\cdot \mathbf u\big]_h\text{~~in ~~} \mathbb R^2\times [0,T].
\end{equation}
Next, Corollary \ref{navier12} implies that $\zeta\varrho\mathbf u$
and $\varrho\nabla\zeta\cdot \mathbf u$ belong to $L^\infty(0,T;
L^1(\mathbb R^2))$. Hence the functions $\div
\big[\zeta\varrho\mathbf u\big]_h$ and $\big[\varrho\nabla\zeta\cdot
\mathbf u\big]_h$ belong to $L^\infty(0,T; C^k(\mathbb R^2))$ for
all integer $k\geq 0$. Moreover, they are supported in $Q_T$. It
follows that   $\partial_t\big[\zeta\varrho\big]_h$ belongs to
$L^\infty(0,T; C^k(\mathbb R^2))$ and is supported in $Q_T$.
Therefore, the function
$$
\partial_t H =\int_{-\infty}^{x_1}\int_{-\infty}^\infty  \partial_t  \big[\zeta\varrho\big]_h(s,x_2,t)\,ds dx_2
$$
belongs to the class $L^\infty(0,T; C^k(\mathbb R^2)).$ Integrating
both sides of \eqref{radon15} over $(-\infty, x_1]\times \mathbb R$
we obtain representation \eqref{radon11} with the reminder
$$
J_0(x_1,t)=\int_{-\infty}^{x_1}\Big\{\int_{\mathbb
R}\big[\varrho\nabla\zeta\cdot \mathbf u\big]_h(s,x_2,t)
dx_2\Big\}ds.
$$
It remains to note that for a. e. $t\in (0,T)$, we have
\begin{multline*}
|J_0(x_1, t)|\leq \int_{\mathbb R^2}\big[\varrho|\nabla\zeta||
\mathbf u|\big]_h(x,t) dx=
\int_{\mathbb R^2}\varrho|\nabla\zeta|| \mathbf u|(x,t) dx\leq\\
c(\zeta)\int_{\Omega}\varrho |\mathbf u|(x,t) dx\leq
c(\zeta)\|\varrho\mathbf u\|_{L^\infty(0,T; L^1(\Omega))}\leq
c(\zeta) E.
\end{multline*}
\end{proof}

In view of Lemma \ref{radonlemma3} the function $\varphi $ given by
formula \eqref{radon9}  belongs to  $L^\infty(0,T; C^k(\mathbb
R^2))$  and is supported in $ Q_T$.  Moreover, we have
\begin{equation}\label{radon16}
    \|\varphi\|_{L^\infty(Q_T)}\leq c(\zeta) \|H\|_{L^\infty(\mathbb R\times(0,T))}\leq c(\zeta) E.
\end{equation}
Substituting $\varphi$ in \eqref{radon4} and using \eqref{radon16}
we obtain
\begin{equation}\label{radon17}
    I_1+I_2+I_3+I_4+I_5\leq c(\zeta)E
\end{equation}
where
\begin{equation}\label{radon18}\begin{split}
    I_1=\int_{\Pi}\varrho u_1\partial_t\varphi dxdt, \quad
    I_2=\int_{\Pi}\varrho u_1u_i\partial_{x_i}\varphi dxdt,\quad
    I_3=\int_{\Pi}p\partial_{x_1}\varphi dxdt, \\
    I_4=-\int_{\Pi}\mathbb S_{i1}\partial_{x_i}\varphi dxdt,\quad
    I_5=\int_{\Pi}\varrho f_1\varphi dxdt.
\end{split}\end{equation}
Let us consider each term in \eqref{radon17} separately.
\begin{lemma}\label{radonlemma4}
\begin{equation}\label{radon21}
I_1=- \int_0^T\int_{\mathbb R} \upsilon_h^2\,dx_1dt +J_1,
\text{~~where~~} |J_1|\leq c(\zeta)E.
\end{equation}
\end{lemma}
\begin{proof}
Since $\varrho u_1\in L^2(\Pi)$ and the mollifying operator is
symmetric, it follows from \eqref{radon9} that
$$
I_1=\int_{\Pi}\Big[\zeta \varrho u_1\Big]_h(x_1,x_2,t)
\partial_t H(x_1,t)\,dxdt.
$$
Next, the function $\Big[\zeta \varrho u_1\Big]_h$ is supported in
$\Omega\times [0,T]$. Therefore, the function $\upsilon_h$ is
supported in every rectangular $[-N,N]\times [0,T]$ such that
$[-N,N]^2\supset \Omega$. From this we conclude that
\begin{equation*}\begin{split}
I_1 =\int_0^T\int_{\mathbb R}\Big\{\int_{\mathbb R}\Big[\zeta
\varrho u_1\Big]_h\, dx_2\Big\}\partial_t H(x_1,t)dx_1dt
=\int_0^T\int_{\mathbb R}\upsilon_h(x_1,t)
\partial_t H(x_1,t)\,dx_1dt.
\end{split}\end{equation*}
 Inserting expression \eqref{radon11} for $\partial_t H$ we obtain representation \eqref{radon21}
with the reminder
$$
J_1=\int_0^T\int_{\mathbb R}J_0 \upsilon_h\, dx_1dt.
$$
It remains to note that in view of \eqref{radon13},
\begin{equation*}\begin{split}
|J_1| \leq c(\zeta)E\int_0^T\int_{\mathbb R}|\upsilon_h|\,dx_1dt\leq
c(\zeta)E \int_\Pi\big[\zeta \varrho |\mathbf u|\big]_h\, dxdt\\=c
E\int_\Pi\zeta \varrho |\mathbf u|\, dxdt=c E\int_{Q_T}\zeta \varrho
|\mathbf u|\, dxdt\leq c(\zeta)E\int_{Q_T}\varrho |\mathbf u|\,
dxdt\leq c(\zeta)E.
\end{split}\end{equation*}
\end{proof}

\begin{lemma}\label{radonlemma5}
\begin{equation}\label{radon22}
I_2=\int_0^T\int_{\mathbb R}\Upsilon_h(x_1,t) \Psi(x_1,t)\, dx_1dt+
J_2,\, \text{where}\, \Upsilon_h=\int\limits_{\mathbb
R}\big[\zeta\varrho u_1^2\big]_h(x_1,x_2, t)\, dx_2,
\end{equation}
and the reminder $J_2$ admits the estimate
\begin{equation}\label{radon23}
    |J_2|\leq c(\zeta) E.
\end{equation}
Moreover, the function $\Upsilon_h$ belongs to the class
$L^\infty(0,T; C^k(\mathbb R))$ for every integer $k\geq 0$. It is
supported in any rectangular $[-N,N]\times [0,T]$ such that
$[-N,N]^2\supset \Omega$.
\end{lemma}
\begin{proof}
Notice that $\zeta\varrho u_iu_1\in L^\infty(0,T; L^1(\Omega))$ is
 supported in $Q_T$.  It follows from
\eqref{radon9} that
\begin{equation*}\begin{split}
    I_2=\int\limits_{\Pi}\zeta\varrho u_1^2\big[\frac{\partial H}{\partial x_1}\big]_h\, dxdt +J_2,\text{~~where~~} J_2=\int\limits_{\Pi}\varrho u_1(\nabla\zeta\cdot\mathbf  u)\big[ H\big]_h\, dxdt.
\end{split}\end{equation*}
Since $\partial_{x_1} H= \Psi$ is independent of $x_2$, we have
\begin{equation*}\begin{split}
\int\limits_{\Pi}\zeta\varrho u_1^2\big[\frac{\partial H}{\partial
x_1}\big]_h\, dxdt= \int\limits_{\Pi}\zeta\varrho
u_1^2\big[\Psi\big]_h\, dxdt=
\int\limits_{\Pi}\big[\zeta\varrho u_1^2\big]_h \Psi\, dxdt=\\
\int_0^T\int_{\mathbb R} \Psi(x_1,t)\Big\{\int_{\mathbb
R}\big[\zeta\varrho u_1^2\big]_h\, dx_2\Big\}
dx_1dt=\int_0^T\int_{\mathbb R}\Upsilon_h(x_1,t) \Psi(x_1,t)\,
dx_1dt.
\end{split}\end{equation*}
This leads to the desired representation \eqref{radon22}. In order
to estimate $J_2$ notice that in view of \eqref{radon10}, we have
$\sup|\big[ H\big]_h|\leq \sup H\leq c(\zeta)E$. This gives
\begin{equation*}\begin{split}
|J_2|\leq c(\zeta) E\int_{Q_T}\varrho |\mathbf u|^2 dxdt\leq
c(\zeta) E.
\end{split}\end{equation*}
\end{proof}

\begin{lemma}\label{radonlemma6}
\begin{equation}\label{radon24}
    I_3\geq \int_0^T \int_{\mathbb R} \Psi^2dx_1dt+J_3, \text{~~where~~} |J_3|\leq c(\zeta, E).
\end{equation}

\end{lemma}
\begin{proof}We have
$$\partial_{x_1}\varphi=\partial_{x_1}\zeta \big[H\big]_h +\zeta \big[ \Psi\big]_h.
$$
Hence
$$
I_3=\int_\Pi\zeta p\big[\Psi\big]_h\, dxdt+J_3, \text{~~where~~}
J_3= \int_\Pi\partial_{x_1}\zeta p\big[H\big]_h\, dxdt.
$$
Since  $\Psi$ is nonnegative we have
\begin{equation*}\begin{split}
\int_\Pi\zeta p\big[\Psi\big]_h\, dxdt=\int_\Pi\zeta
\varrho\big[\Psi\big]_h\, dxdt+\varepsilon\int_\Pi\zeta
\varrho^\gamma\big[\Psi\big]_h\, dxdt
\geq \\
\int_\Pi\zeta \varrho\big[\Psi\big]_h\, dxdt= \int_\Pi\big[\zeta
\varrho\big]_h \Psi\, dxdt=\int_0^T\int_{\mathbb R} \Psi^2dx_1dt,
\end{split}\end{equation*}
which leads to  \eqref{radon24}. It remains to estimate $J_3$. To
this end notice that in view of \eqref{radon10} and \eqref{navier7},
$$
|J_3|\leq c(\zeta) \|H\|_{L^\infty(\mathbb R\times (0,T))}
\int_0^T\int_\Omega p \,dxdt\leq c(\zeta) E.
$$
\end{proof}

\begin{lemma}\label{radonlemma7}
\begin{gather}\label{radon26}
    |I_4|\leq  c(\zeta)E+c(\zeta)\,E\,\Big(\int_0^T\int_{\mathbb R}\Psi^2 dx_1dt\Big)^{1/2}, \quad
    |I_5|\leq c(\zeta)E.
\end{gather}

\end{lemma}
\begin{proof} It follows from formulae \eqref{radon9} that
\begin{equation}\label{radon27}
I_4= \int_\Pi\zeta \mathbb S_{11}(\mathbf u)\big[\Psi\big]_h dxdt+
\int_\Pi(\partial_{x_i}\zeta )\,\mathbb S_{i1}(\mathbf
u)\big[H\big]_h dxdt.
\end{equation}
Notice that  $[\zeta \mathbb S]_h$ is compactly supported in
$\Omega\times [0,T]$. Hence it is supported in the slab
$[-N,N]^2\times[0,T]$ such that $[-N,N]^2\supset\Omega$. Thus we get
\begin{equation*}\begin{split}
\Big|\int_\Pi \zeta \mathbb S_{11}\big[\Psi\big]_h dxdt\Big|=
\Big|\int_\Pi\big[\zeta \mathbb S_{11}\big]_h \Psi dxdt\Big|=
\Big|\int_0^T\int_{[-N,N]^2} \big[\zeta \mathbb S_{11}\big]_h \Psi dxdt\Big|\leq\\
\Big(\int_0^T\int_{[-N,N]^2} \big[\zeta \mathbb S_{11}\big]_h^2
dxdt\Big)^{1/2}
\Big(\int_0^T\int_{[-N,N]^2}  \Psi^2 dxdt\Big)^{1/2}\leq\\
\Big(\int_0^T\int_{[-N,N]^2} (\zeta \mathbb S_{11})^2
dxdt\Big)^{1/2} \Big(N\int_0^T\int_{[-N,N]}  \Psi^2
dx_1dt\Big)^{1/2}\leq\\ c(\zeta)N^{1/2} \Big(\int_{Q_T}
|\nabla\mathbf u|^2  dxdt\Big)^{1/2} \Big(\int_0^T\int_{\mathbb R}
\Psi^2 dx_1dt\Big)^{1/2}
\end{split}\end{equation*}
Since $N$ depends only on $\Omega$, these inequalities along with
estimate \eqref{navier7} imply
\begin{equation}\label{radon28}
\Big|\int_\Pi \zeta \mathbb S_{11}\big[\Psi\big]_h dxdt\Big|\leq
c(\zeta)E \Big(\int_0^T\int_{\mathbb R}  \Psi^2 dx_1dt\Big)^{1/2}
\end{equation}
We   also have
\begin{equation}\begin{split}\label{radon29}
\Big|\int_\Pi(\partial_{x_i}\zeta)\, \mathbb S_{i1}(\mathbf
u)\big[H\big]_h dxdt\Big| \leq c(\zeta)\| H\|_{L^\infty(\mathbb
R\times (0,T))} \int_{Q_T}|\nabla\mathbf u| dxdt \leq c(\zeta) E.
\end{split}\end{equation}
Inserting  \eqref{radon28} and \eqref{radon29} into \eqref{radon27}
we arrive at the first inequality in \eqref{radon26}. It remains to
notice that
$$
|I_5|\leq c \|f\|_{L^\infty(Q_T)}\| H\|_{L^\infty(\mathbb R\times
(0,T))}\int_\Pi \zeta\varrho\,dxdt \leq c(\zeta)E \int_{Q_T}
\varrho\,dxdt\leq c(\zeta) E.
$$
\end{proof}

\begin{lemma}\label{radonlemma8}
\begin{gather}\label{radon30}
    \frac{1}{2}\int_0^T\int_{\mathbb R}\Psi^2 dx_1dt+\int_0^T\int_{\mathbb R}(\Upsilon_h
    \Psi-\upsilon_h^2)\, dx_1dt
\leq c(\zeta)E.
\end{gather}
\end{lemma}
\begin{proof} Estimate \eqref{radon26} for $I_5$ and inequality \eqref{radon17} imply
$$
I_1+I_2+I_3+I_4\leq c(\zeta) E.
$$
From this and  Lemmas \ref{radonlemma4}, \ref{radonlemma5} we obtain
$$
I_3+\int_0^T\int_{\mathbb R}(\Upsilon_h
    \Psi-\upsilon_h^2)\, dx_1dt\leq c(\zeta)E+|I_4|
$$
Applying Lemmas \ref{radonlemma6} and \ref{radonlemma7} we arrive at
$$
\int_0^T \int_{\mathbb R} \Psi^2dx_1dt+\int_0^T\int_{\mathbb
R}(\Upsilon_h
    \Psi-\upsilon_h^2)\, dx_1dt\leq c(\zeta)E+c(\zeta)\,E\,\Big(\int_0^T\int_{\mathbb R}\Psi^2 dx_1dt\Big)^{1/2},
$$
which  obviously leads to  \eqref{radon30}.
\end{proof}
\begin{lemma}\label{radonlemma9} $\Upsilon_h
    \Psi-\upsilon_h^2\geq 0$ in $\Pi$.
\end{lemma}
\begin{proof} We begin with the observation that the inequality
$$
\big[ fg\big]_h^2\leq \big[ f^2\big]_h\, \big[ g^2\big]_h
$$
holds for all functions $f(x)$, $g(x)$ locally integrable with
square in $\mathbb R^2$. Setting $f=\sqrt{\zeta\varrho(\cdot,t)}$
and $g=\sqrt{\zeta\varrho}| u_1|(\cdot,t)$ we obtain for a.e. $x,t$
and all $\delta>0$,
\begin{equation*}\begin{split}
\big[ \zeta\varrho |u_1|\big]_h(x_1,x_2,t)\leq
\sqrt{\big[\zeta\varrho u_1^2\big]_h}(x_1,x_2,t)
\sqrt{\big[\zeta\varrho\big]_h}(x_1,x_2,t)\leq\\
\frac{1}{2}\delta\big[\zeta\varrho
u_1^2\big]_h(x_1,x_2,t)+\frac{1}{2}\delta^{-1}\big[\zeta\varrho\big]_h(x_1,x_2,t)
\end{split}\end{equation*}
Integrating both sides with respect to $x_2$ over $\mathbb R$ and
recalling formulae \eqref{radon11} and \eqref{radon22} for
$\upsilon_h$ and $\Upsilon_h$ we arrive at
\begin{equation*}\begin{split}
\upsilon_h(x_1,t)\leq \frac{1}{2}\delta
\Upsilon_h(x_1,t)+\frac{1}{2}\delta^{-1}\Psi(x_1,t).
\end{split}\end{equation*}
Recall that $\Upsilon_h$ and $\Psi$ are nonnegative. Setting
$\delta=(\Psi/\Upsilon_h)^{1/2}$ we obtain the desired inequality.
\end{proof}

We are now in a position to complete the proof of Theorem
\ref{radontheorem}. Introduce the function
\begin{equation*}
    \Phi_1(x_1,t)=\int_\mathbb R \zeta\varrho(x_1,x_2,t)\, dx_2\equiv \Phi(\mathbf e_1, x_1, t).
\end{equation*}
Recall that $\zeta \varrho$ is  supported in $Q_T$. It suffices to
prove that
\begin{equation}\label{radon34}
    \int_0^T\int_{\mathbb R} \Phi_1^2 \, dx_1dt\leq c(\zeta)E.
\end{equation}
By virtue of the energy estimate \eqref{navier7}, the function
$\Phi_1$ belongs to the class $L^2(0,T; \mathbb R)$. It is supported
in the  rectangular $[-N,N]\times [0,T]$ for every $N$ such that
$[-N,N]^2\supset\Omega$. It  obviously follows from this and
definition \eqref{radon7} of the mollifier that
\begin{equation*}
\Psi=\big[\Phi_1\big]^{(1)}_h, \text{~~where~~}
\big[\Phi_1\big]^{(1)}_h(x_1,t)=\frac{1}{h}\int_\mathbb R
\omega\Big(\frac{x_1-y_1}{h}\Big)\Phi_1(y_1,t)dy_1.
\end{equation*}
In other words, $\big[\Phi_1\big]^{(1)}_h$ is the mollifying of
$\Phi_1$ with respect to $x_1$. Lemmas \ref{radonlemma8} and
\ref{radonlemma9} imply the inequality
\begin{equation}\label{radon36}
 \int_0^T\int_{\mathbb R} \big(\,\big[\Phi_1\big]^{(1)}_h\,\big)^2\, dx_1dt\leq c(\zeta)E.
\end{equation}
Notice that $\big[\Phi_1\big]^{(1)}_h\to \Phi_1$ a.e. in $\mathbb
R\times (0,T)$. Letting $h\to 0$ in \eqref{radon36} and applying the
Fatou Theorem we arrive at \eqref{radon34}. $\square$

\section{ Momentum estimates}\label{momentum}
In this section we prove  auxiliary estimates for solutions
$(\varrho, \mathbf u)$ to regularized equations \eqref{navier1}.  We
start with the estimating of  norms of $\varrho$ and $\varrho
\mathbf u$ in negative Sobolev spaces.
\begin{proposition} \label{momentproposition}Let a solution $(\varrho, \mathbf u)$ to problem \eqref{navier1} meets all requirements of Proposition \ref{navierproposition}. Let $\zeta\in C^\infty_0(\mathbb R^2)$
be an arbitrary nonnegative compactly supported in $\Omega$ function
and $s>1/2$. Then
\begin{gather}\label{moment1}
    \|\zeta\varrho\|_{L^2(0,T; H^{-1/2}(\mathbb R^2))}\leq c(\zeta)E,\\
  \label{moment2}
  \|\zeta\varrho\mathbf u\|_{L^1(0,T; H^{-s}(\mathbb R^2))}\leq c(\zeta)c(s) E,
\end{gather}
where  $c(\zeta)$ depends only on $\zeta$, $c(s)$ depends only on
$s$, and $E$ is specified by Remark \ref{stokesremark}.
\end{proposition}
\begin{proof}Lemma \ref{sobolevlemma1} and Theorem \ref{radontheorem} imply the estimates
\begin{equation*}\begin{split}
    \int_0^T \|\zeta\varrho\|_{H^{-1/2}(\mathbb R^2)}^2\, dt\leq
\int_0^T\int_{\mathbb
S^1}\int_{-\infty}^{\infty}\Big\{\int\limits_{\boldsymbol \omega
\cdot x=\tau} \zeta\varrho\, dl\Big\}^2
d\tau d\boldsymbol\omega dt=\\
 \int_{\mathbb S^1}\Big\{\int_0^T\int_{-\infty}^{\infty}\Big\{\int\limits_{\boldsymbol \omega \cdot x=\tau} \zeta\varrho\, dl\Big\}^2
d\tau dt \Big\}d\boldsymbol\omega \leq c(\zeta)E\int_{\mathbb
S^1}d\boldsymbol\omega\leq c(\zeta) E,
\end{split}\end{equation*}
which yield \eqref{moment1}. Next,  Lemma \ref{sobolevlemma2}
implies the inequality
$$\|\zeta\varrho(t)\mathbf u(t)\|_{H^{-s}(\mathbb R^2)}\leq c(\zeta,s)\|
\zeta\varrho(t)\|_{H^{-1/2}(\mathbb R^2)}\|\mathbf
u(t)\|_{H^{1}(\mathbb R^2)}.
$$
From this, \eqref{moment1}, and estimate \eqref{navier7} we finally
obtain
\begin{equation*}\begin{split}
\|\zeta\varrho\mathbf u\|_{L^1(0,T;H^{-s}(\mathbb R^2))}\leq
c(\zeta,s)\| \zeta\varrho\|_{L^2(0,T;H^{-1/2}(\mathbb
R^2))}\|\mathbf u\|_{L^2(0,T;H^{1}(\mathbb R^2))} \leq c(\zeta,s) E.
\end{split}\end{equation*}
\end{proof}
\subsection{Cauchy-Riemann equations}
Further notation  $\nabla^\perp$ and $\text{rot}$ stands for the
differential operators
$$
\nabla^\perp f=(-\partial_{x_2}f, \partial_{x_1} f), \quad
\text{rot~}\mathbf w=\partial_{x_2}w_1-\partial_{x_1}w_2.
$$
Denote by
 $\mathbf F=(F_1, F_2)$ a solution to the inhomogeneous Cauchy-Riemann equation
\begin{equation}\label{moment3}
     \nabla F_1+\nabla^\perp F_2= \zeta\varrho \mathbf u\text{~~in~~} \Pi.
\end{equation}
It is easily seen that
\begin{equation}\label{riesz3}
    F_1=\text{div~} \Delta^{-1} (\zeta\varrho\mathbf u), \quad F_2=-\text{rot~} \Delta^{-1}
    (\zeta\varrho\mathbf u).
\end{equation}
The following two auxiliary lemmas give  $L^p$- estimates for a
vector function $\mathbf F$.
\begin{lemma}\label{momentlemma1} Under the assumptions of Proposition
\ref{momentproposition}, for every positive $\delta$ and $R$, there
is a constant $c(\delta, \zeta, R)$ such that
\begin{equation}\label{moment4}
    \|\mathbf F\|_{L^4(0,T; L^{\frac{8}{3+\delta}}(B_R))}\leq c(\delta, \zeta, R)E.
\end{equation}
\end{lemma}
\begin{proof} Fix an arbitrary positive $\delta$ and $R$. Without loss of generality we may assume that $\delta<1$ and $B_R\supset \Omega$. Integral representation \eqref{riesz1} for the operator $\partial_x\Delta^{-1}$
and formulae \eqref{riesz3} for solutions to inhomogeneous
Cauchy-Riemann equations imply the inequalities
\begin{equation*}
|\mathbf F(x,t)|\leq c\int_\Omega|x-y|^{-1}\zeta\varrho |\mathbf
u|(y,t)\,dy\leq \frac{c}{2}b(x,t)\mathcal L(x,t) +\frac{c}{2}
b^{-1}(x,t) \mathcal Q(x,t),
\end{equation*}
where
$$
\mathcal L(x,t)=\int_\Omega|x-y|^{-1}\zeta\varrho |\mathbf
u|^2(y,t)\,dy,\quad \mathcal
Q(x,t)=\int_\Omega|x-y|^{-1}\zeta\varrho\,dy,
$$
$b$ is an arbitrary positive function. Notice that if $\mathcal
L(x,t)$ or $\mathcal Q(x,t)$ vanishes at least at one point $(x,t)$,
then $\mathbf F(\cdot,t)$ vanishes in $\mathbb R^2$. In opposite
case we can take $b=\sqrt{\mathcal Q/\mathcal L}$. Thus we get
\begin{equation}\label{moment5}
|\mathbf F(x,t)|\leq c\sqrt{\mathcal L(x,t)}\, \sqrt{\mathcal
Q(x,t)} \text{~~a.e. in~~} \mathbb R^2\times (0,T)
\end{equation}
Now our task is to estimate $\mathcal L$ and $\mathcal Q$. We have
$$
\mathcal L(x,t)=\int_{B_R} \big(\zeta\varrho |\mathbf
u|^2\big)^{\frac{1+\delta}{2}}\,|x-y|^{-1-\delta+\alpha}\,
\big(\zeta\varrho |\mathbf
u|^2\big)^{\frac{1-\delta}{2}}\,|x-y|^\alpha\, dy,
$$
where $\alpha =\delta/2>0$. It follows that
$$
\mathcal L(x,t)\leq c(R)\int_{B_R} \big(\zeta\varrho |\mathbf
u|^2\big)^{\frac{1+\delta}{2}}\,|x-y|^{-1-\delta+\alpha}
\big(\zeta\varrho |\mathbf u|^2\big)^{\frac{1-\delta}{2}}\,
dy\text{~~for~~} x\in B_R.
$$
Applying the H\"{o}lder inequality we obtain
$$
\mathcal L(x,t)\leq c\Big(\int_{B_R} \zeta\varrho |\mathbf u|^2
|x-y|^{-2+\frac{2\alpha}{1+\delta}}dy
\Big)^{\frac{1+\delta}{2}}\Big(\int_{B_R} \zeta\varrho |\mathbf
u|^2\,dy \Big)^{\frac{1-\delta}{2}}.
$$
 This leads to the inequality
\begin{equation*}\begin{split}
\int_{B_R} \mathcal L(x,t)^{\frac{2}{1+\delta}}dx\leq
\Big(\int_{B_R} \zeta\varrho |\mathbf u|^2\,dy
\Big)^{\frac{1-\delta}{1+\delta}}\int_{B_R}\int_{B_R}\zeta\varrho |\mathbf u|^2(y,t)|x-y|^{-2+\frac{2\alpha}{1+\delta}}\, dxdy\leq \\
c\Big(\int_{B_R} \zeta\varrho |\mathbf u|^2(y,t)\,dy
\Big)^{\frac{2}{1+\delta}}\leq c\big(\|\varrho |\mathbf
u|^2(t)\|_{L^1(\Omega)}\big)^{2/(1+\delta)}.
\end{split}\end{equation*}
Recalling the energy estimate \eqref{navier7} we finally obtain
\begin{equation}\label{moment7}
\|\mathcal L\|_{L^\infty(0,T;L^{\frac{2}{1+\delta}}(B_R))}\leq
c(R,\zeta) \|\varrho |\mathbf u|^2\|_{L^\infty(0,T;L^1(\Omega))}\leq
c(\zeta, R)E.
\end{equation}
Now our task is to estimate $\mathcal Q$. Notice that $|x-y|\leq 2
R$ for all $x,y\in B_R$.  It follows from this and \eqref{sobolev4}
that $|x-y|^{-1}\leq c(R) G_1(x-y)$ for all $x,y\in B_R$, where
$G_1$ is the Bessel kernel. We thus get
\begin{equation*}
    \mathcal Q(x,t)\leq c(R)
    \int_{\mathbb R^2}G_1(x-y)\zeta\varrho(y,t)\, dy:= G(x,t).
\end{equation*}
Estimate  \eqref{sobolev5}  for the Bessel kernel and
 inequality \eqref{moment1} yield
$$\|G\|_{L^2(0,T; H^{1/2}(\mathbb R^2)}\leq \|\zeta\varrho\|_{L(0,T; H^{-1/2}(\mathbb R^2))}
\leq c(\zeta)E.
$$Since the embedding $H^{1/2}(\mathbb R^2)\hookrightarrow L^4(\mathbb R^2)$
is bounded, see \cite{adams}, thm. 7.57, we obtain
\begin{equation*}
    \|\mathcal Q\|_{L^2(0,T; L^4(B_R))}\leq c(R)\|G\|_{L^2(0,T; L^4(B_R))}
    \leq c(R)\|G\|_{L^2(0,T; H^{1/2}(\mathbb R^2))}\leq c(R, \zeta)E.
\end{equation*}
Combining these inequalities with \eqref{moment7}  we arrive at the
estimates
\begin{equation}\label{moment10}
    \|\sqrt{\mathcal L}\|_{L^\infty(0,T;L^{\frac{4}{1+\delta}}(B_R))}\leq
    c(R,\zeta)E, \quad \|\sqrt{\mathcal Q}\|_{L^4(0,T;L^{8}(B_R))}\leq
    c(R,\zeta)E.
\end{equation}
Next, the H\"{o}lder inequality implies that
$$
\|\sqrt{\mathcal L}\sqrt{\mathcal Q}\|_{L^\tau(0,T;L^{r}(B_R))}\leq
\|\sqrt{\mathcal L}\|_{L^{\tau_1}(0,T;L^{r_1}(B_R))}
\|\sqrt{\mathcal Q}\|_{L^{\tau_2}(0,T;L^{r_2}(B_R))}
$$
for all $\tau,r, \tau_i,r_i \in[1,\infty]$ satisfying the condition
$$
\tau^{-1}=\tau_1^{-1}+\tau_2^{-1}, \quad r^{-1}= r_1^{-1}+r_2^{-1}.
$$
Setting $\tau=\tau_2=4$, $\tau_1=\infty$, $r_1=4/(1+\delta)$,
$r_2=8$, $r=8/(3+2\delta)$, and recalling inequalities
\eqref{moment10} we obtain
$$
\|\sqrt{\mathcal L}\sqrt{\mathcal
Q}\|_{L^4(0,T;L^{8/(3+2\delta)}(B_R))}\leq c(\zeta,\delta, R) E.
$$
Combining this result with \eqref{moment5} we arrive at
\eqref{moment4}.
\end{proof}

\begin{lemma}\label{momentlemma2}. Under the assumptions of
Proposition \ref{momentproposition}, for every positive $\nu<3$ and
$R$, there is a constant $c(\nu, \zeta, R)$ such that
\begin{equation}\label{moment11}
    \|\mathbf F\|_{L^2(0,T; L^{3-\nu}(B_R))}
    \leq c(\nu, \zeta, R)E.
\end{equation}
\end{lemma}
\begin{proof} Assume that $\Omega\subset B_R$. It follows from
\eqref{sobolev4} that $|x-y|^{-1}\leq c(R) G_1(x-y)$ for all $x,
y\in B_R$. Thus we get
\begin{equation}\label{moment12}
|\mathbf F(x,t)|\leq c\int_{\mathbb R^2} G_1(x-y)\zeta\varrho
|\mathbf u|(y,t)\,dy:= M(x,t)\text{~~for all~~} x\in B_R.
\end{equation}
Now choose an arbitrary  $\mu\in (0,1/2)$. It follows from
\eqref{sobolev5} that
$$
\|M(t)\|_{H^{1/2-\mu}(\mathbb R^2)}\leq \|\zeta\varrho |\mathbf
u|(t)\|_{H^{-1/2-\mu}(\mathbb R^2)}.
$$
Applying inequality \eqref{moment2} we obtain
$$
\|M\|_{L^1(0,T;H^{1/2-\mu}(\mathbb R^2))}\leq c(\mu,\zeta)E.
$$
Since the embedding $H^{1/2-\mu}(\mathbb R^2)\hookrightarrow
L^{\frac{4}{1+2\mu}}(\mathbb R^2)$ is bounded,   \cite{adams}, thm.
7.57, we get
$$
\|M\|_{L^1(0,T;L^{\frac{4}{1+2\mu}}(\mathbb R^2))}\leq
c(\mu,\zeta)E.
$$
Combining this result with \eqref{moment12} we arrive at
\begin{equation}\label{moment13}
\|\mathbf F\|_{L^1(0,T;L^{\frac{4}{1+2\mu}}(B_R))}\leq
c(\mu,\zeta)E.
\end{equation}
Next notice that by the interpolation inequality,
$$
\|\mathbf F\|_{L^r(0,T; L^s(B_R))}\leq \|\mathbf F\|_{L^4(0,T;
L^{\frac{8}{3+\delta}}(B_R))}^\alpha\|\mathbf F\|_{L^1(0,T;
L^{\frac{4}{1+2\mu}}(B_R))}^{1-\alpha}
$$
holds for all $\alpha\in (0,1)$ and
$$
r^{-1}=\frac{\alpha}{4}+\frac{1-\alpha}{1}, \quad s^{-1}=
\frac{3+\delta}{8} \alpha+\frac{1+2\mu}{4} (1-\alpha).
$$
Setting $\alpha=2/3$ and recalling inequalities \eqref{moment4},
\eqref{moment13} we obtain
$$
\|\mathbf F\|_{L^2(0,T; L^{\frac{12}{4+\mu+\delta}}(B_R))}\leq
c(\zeta, \delta, \mu, R).
$$
Choosing $\mu $ and $\delta$ so small that $3-\nu<12/(4+\mu+\delta)$
we finally arrive at \eqref{moment11}.

\end{proof}

\section{$L^p$ estimates}\label{lpestimate}
In this section we investigate properties of   solutions $(\varrho,
\mathbf u)$ to regularized problem \eqref{navier1} and prove that
the pressure function $p(\varrho)$ is locally  integrable with an
exponent greater than $1$. The corresponding result is given by the
following

\begin{theorem}\label{lebegtheorem}
Let a solution $(\varrho, \mathbf u)$ to problem \eqref{navier1}
meets all requirements of Proposition \ref{navierproposition}. Let
$\zeta\in C^\infty_0(\mathbb R^2)$ be an arbitrary nonnegative
compactly supported in $\Omega$ function and $\lambda\in (0,1/6)$.
Then
\begin{gather}\label{lebeg1}
    \int_{Q_T}\zeta^2 p(\varrho)\varrho^\lambda\, dxdt\leq c(\zeta, \lambda)E,
\end{gather}
where  $c(\zeta,\lambda)$ depends only on $\zeta$ and $\lambda$.
\end{theorem}

The rest of the section is devoted to the proof of this theorem. Our
strategy is the following. First we construct a special  test
function $\boldsymbol \xi$ such that $
p\text{~div~}\boldsymbol\xi\sim p(\varrho)\varrho^\lambda$. Next we
insert $\boldsymbol\xi$ into \eqref{navier8} to obtain special
integral identity containing the vector field $\mathbf F$. Finally
we employ Lemmas \ref{momentlemma1} and \ref{momentlemma2} to obtain
estimate \eqref{lebeg1}. Hence the proof of Theorem
\ref{lebegtheorem} falls into four steps.

\subsection{Step 1. Test functions}
Fix an arbitrary $\lambda\in (0,1/6)$ and choose a function $\psi\in
C^\infty(\mathbb R)$ with the properties
\begin{equation}\label{lebeg2}
     \psi(0)=0,\quad \psi(s)\geq
    0,\quad c^{-1}|s|^\lambda-1\leq \psi(s)\leq c|s|^\lambda, \quad
    |s\psi'(s)|\leq c |s|^\lambda,
\end{equation}
where $c$ is some positive constant. Next choose an arbitrary
function $\zeta\in C^{\infty}_0(\mathbb R^2)$ such that $\zeta$ is
nonnegative and is compactly supported in $\Omega$. Recall the
definition of the mollifier $\big[\cdot\big]_h$ and introduce the
auxiliary function
\begin{equation}\label{lebeg7}
g(x,t)=\big[\zeta\psi(\varrho)\big]_h(x,t)\text{~~in~~} R^2\times
[0,T].
\end{equation}
We will assume that $h$ is less than the distance between  the
support of $\zeta$ and the boundary of $\Omega$. Finally, introduce
the test vector field
\begin{equation}\label{lebeg11}
    \boldsymbol \xi(x,t)=\zeta(x)\, \mathbf H(x,t),\text{~~where~~} \mathbf H=\nabla\Delta^{-1} g.
\end{equation}
The following lemmas constitute the basic properties of
$\psi(\varrho)$, $g$, and $\mathbf H$.
\begin{lemma}\label{lebeglemma1}
Under the assumptions of Theorem \ref{lebegtheorem}, there is a
constant $c(\lambda)$, depending only on $\lambda$ and $\psi$, such
that
\begin{equation}\label{lebeg3}
\|\psi(\varrho)\nabla\mathbf u\|_{L^2(0,T;
L^{\frac{2}{1+2\lambda}}(\mathbb R^2))}+
\|(\psi(\varrho)-\varrho\psi'(\varrho))\nabla\mathbf u\|_{L^2(0,T;
L^{\frac{2}{1+2\lambda}}(\mathbb R^2))}\leq c(\lambda)E,
\end{equation}
\begin{equation}\label{lebeg4}
\|\psi(\varrho)\mathbf u\|_{L^2(0,T; L^{3}(\mathbb R^2))}\leq
c(\lambda)E.
\end{equation}
\end{lemma}
\begin{proof} Notice that
\begin{equation}\label{lebeg5}
|\psi(\varrho)\nabla\mathbf
u|+|(\psi(\varrho)-\varrho\psi'(\varrho))\nabla\mathbf u|\leq c
\varrho^{\lambda}|\nabla\mathbf u|.
\end{equation}
Recall that  $\mathbf u$ and $\varrho$ vanish outside of
$\Omega\times[0,T]$. From this and relations
$$
1/2+1/\infty=1/2,\quad 1/2+1/(1/\lambda)=(1+2\lambda)/2
$$
we obtain
\begin{equation*}\begin{split}
\|\varrho^\lambda\nabla\mathbf u\|_{L^2(0,T;
L^{\frac{2}{1+2\lambda}}(\Omega))}\leq  \|\nabla\mathbf
u\|_{L^2(0,T; L^2(\Omega))}\, \|\varrho^\lambda \|_{L^\infty(0,T;
L^{1/\lambda}(\Omega))}\leq \\c(\lambda) \|\nabla\mathbf
u\|_{L^2(0,T; L^2(\Omega))}\, \|\varrho \|_{L^\infty(0,T;
L^{1}(\Omega))}^\lambda\leq c(\lambda) E,
\end{split}\end{equation*}
which along with \eqref{lebeg5} yields \eqref{lebeg3}. Next set
$q=3/(1-3\lambda)$. Since $1/q+1/(1/\lambda)=1/3$, we have
\begin{equation*}\begin{split}
\|\psi(\varrho)\mathbf u\|_{L^2(0,T; L^{3}(\mathbb R^2))}\leq
c\|\varrho^\lambda\mathbf u\|_{L^2(0,T; L^{3}(\Omega))}\leq
c\|\varrho^\lambda\|_{L^\infty(0,T;
L^{1/\lambda}(\Omega))}\|\mathbf u\|_{L^2(0,T; L^{q}(\Omega))}\\
\leq c(\lambda)\|\varrho\|_{L^\infty(0,T;
L^{1}(\Omega))}^\lambda\|\mathbf u\|_{L^2(0,T; L^{q}(\Omega))}\leq
c(\lambda)\|\varrho\|_{L^\infty(0,T;
L^{1}(\Omega))}^\lambda\|\mathbf u\|_{L^2(0,T;
W^{1,2}_0(\Omega))}\leq cE,
\end{split}\end{equation*}
and the lemma follows.
\end{proof}
\begin{lemma}\label{lebeglemma2}
Under the assumptions of Theorem \ref{lebegtheorem}, the function
$g$ belongs to the class $L^\infty (0,T; C^k(\mathbb R^2))$ for
every integer $k\geq 0$. It is compactly supported in $\Omega\times
[0,T]$ and admits the estimate
\begin{equation}\label{lebeg8}
    \|g\|_{L^\infty(0,T; L^{1/\lambda}(\mathbb R^2))}\leq c(\zeta) E.
\end{equation}
Moreover,  $\partial_t g$ belongs to the class $L^2 (0,T;
C^k(\mathbb R^2))$  and has the representation
\begin{equation}\label{lebeg9}
    \partial_t g=-\text{\rm div~} \big[ \zeta\psi(\varrho)\mathbf u]_h+
   \big[\psi(\varrho)\nabla\zeta\mathbf u]_h+
   \big[ \zeta(\psi(\varrho)-\psi'(\varrho)\varrho)\text{\rm ~div~}\mathbf u]_h.
\end{equation}
\end{lemma}
\begin{proof} Since $\zeta\psi(\varrho)\in L^\infty(0,T; L^1(\mathbb R^2))$, it follows from general properties of the mollifier that $g\in L^\infty
(0,T; C^k(\mathbb R^2))$. Since $h$ is less than the distance
between $\text{~spt~}\zeta$ and $\mathbb R^2\setminus \Omega$, the
function $g$ is supported in $\Omega\times[0,T]$. Next, inequality
\eqref{lebeg2} implies the estimate
$$
\|g(t)\|_{L^{1/\lambda}(\mathbb R^2)}\leq
\|\zeta\psi(\varrho)(t)\|_{L^{1/\lambda}(\mathbb R^2)}\leq
c(\zeta)\|\varrho(t)\|_{L^{1}(\mathbb R^2)}^{\lambda},
$$
which along with energy inequality \eqref{navier7} yields
\eqref{lebeg8}. Let us consider the time derivative of $g$. In view
of \eqref{navier10} the integral identity
\begin{multline}\label{lebeg10}
\int_{Q_T}\Big(\psi(\varrho)\partial_t\varsigma+
\big(\psi(\varrho)\mathbf u\big)
\cdot\nabla\varsigma-\varsigma\big(\psi'(\varrho)\varrho-\psi(\varrho)\big)\div\mathbf
u\Big)\,dxdt=0
\end{multline}
holds  for every function $\varsigma\in C^\infty(\Pi)$ which is
supported in $Q_T$. Choose an arbitrary $\xi\in C^\infty_0(0,T)$,
$y\in \mathbb R^2$, and set
$$
\varsigma(x,t)=\xi
(t)\zeta(x)\,h^{-2}\omega\big(\frac{x_1-y_1}{h}\big)
    \omega\big(\frac{x_2-y_2}{h}\big).
$$
Substituting $\varsigma$ into \eqref{lebeg10} we obtain
\begin{multline*}
\int_0^T\xi'(t) g(y,t)\,dt -\int_0^T\xi(t)\text{div~} \big[
\zeta\psi(\varrho)\mathbf u]_h(y,t)\, dt +\int_0^T\xi(t)
\big[\psi(\varrho)\nabla\zeta\mathbf u]_h(y,t)\, dt+\\\int_0^T\xi(t)
\big[ \zeta(\psi(\varrho)-\psi'(\varrho)\varrho)\text{~div~}\mathbf
u]_h(y,t)\,dt=0,
\end{multline*}
which yields \eqref{lebeg9}. In view of Lemma \ref{lebeglemma1}, the
functions $\psi(\varrho)\zeta\mathbf u$,~
$\psi(\varrho)\nabla\zeta\cdot \mathbf u$ belong to the class
$L^2(0,T; L^3(\mathbb R^2)$, and the function
$\zeta(\psi(\varrho)-\psi'(\varrho)\varrho)\text{~div~}\mathbf u$
belong to $L^2(0,T; L^{2/(1+2\lambda)}(\mathbb R^2)$. Since the
mollifier $\big[\cdot\big]_h :L^p(\mathbb R^2)\to C^k(\mathbb R^2)$
is bounded for all $p\geq 1$ and $k\geq 0$, the representation
\eqref{radon9} yields the inclusion $\partial_t g\in L^2 (0,T;
C^k(\mathbb R^2))$ .
\end{proof}
\begin{lemma}\label{lebeglemma3} Under the assumptions of Theorem \ref{lebegtheorem},
$\mathbf H$ belongs to the class class $L^\infty (0,T; C^k(\mathbb
R^2))$, and  $\partial_t \mathbf H$ belongs to the class $L^2 (0,T;
C^k(\mathbb R^2))$ for every integer $k\geq 0$. Moreover, $\mathbf
H$ admits the estimates.
\begin{equation}\label{lebeg12}
  \|\mathbf H\|_{L^\infty(0,T; L^\infty(\Omega))}\leq c(\zeta) E,\quad
\|\nabla \mathbf H\|_{L^\infty(0,T; L^{1/\lambda}(\Omega))}\leq
c(\zeta) E.
\end{equation}

\end{lemma}

\begin{proof}
Since the function $g$ is supported in $Q_T$, the inclusions
$\mathbf H\in L^\infty (0,T; C^k(\mathbb R^2))$  and $\partial_t
\mathbf H\in L^2 (0,T; C^k(\mathbb R^2))$ obviously follow from
\eqref{lebeglemma2}. Now choose an arbitrary $R$ such that
$\Omega\Subset B_{R/2}$. Since $\text{spt } g(t)\subset \Omega$, we
have
$$
\|\mathbf H(t)\|_{W^{1/\lambda}(B_R)}\leq c(R, \Omega)
\|g(t)\|_{L^{1/\lambda}(\Omega)}
$$
The embedding $W^{1/\lambda}(B_R)\hookrightarrow C(B_R)$ is bounded
for every $\lambda<1/2$. It follows that
\begin{multline}\label{lebeg13}
\|\mathbf H\|_{L^\infty(0,T; C(B_R))}\leq c(R,\lambda)\|\mathbf
H\|_{L^\infty(0,T;W^{1/\lambda}(B_R))}\leq\\
\|g\|_{L^\infty(0,T;L^{1/\lambda}(\Omega)}\leq c(R, \zeta, \lambda)
E
\end{multline}
 Since $\Omega\subset B_{R/2}$, representation \eqref{riesz1} implies
$$
|\mathbf H(x,t)|\leq c(R, \Omega)
\|g(t)\|_{L^{1}(\Omega)}\text{~~for~~}x\in \mathbb R^2\setminus B_R
$$
Thus we get
\begin{equation*}
  \|\mathbf H\|_{L^\infty(0,T; C(\mathbb R^2\setminus B_R))}\leq c(R,\lambda)
\|g\|_{L^\infty(0,T;L^{1}(\Omega)}\leq c(R, \zeta, \lambda) E.
\end{equation*}
Combining this result with \eqref{lebeg13} gives the first
inequality in \eqref{lebeg12}. Next notice that
$$
\partial_{x_i}H_j=\partial_{x_i}\partial_{x_j}\Delta^{-1} g= R_iR_j g,
$$
where $R_i$ , $R_j$ are the Riesz singular operators. Since the
Riesz operators are bounded in $L^{1/\lambda}(\mathbb R^2)$, the
second inequality in \eqref{lebeg12} is a straightforward
consequence of estimate \eqref{lebeg8}.

\end{proof}

\subsection{Step 2. Integral identities }
The proof of Theorem \ref{lebegtheorem} is based on the special
integral identity which is given by the following proposition.
\begin{proposition}\label{lebegproposition}
Under the assumptions of Theorem \ref{lebegtheorem}, we have
\begin{equation}\label{lebeg15}
    \int_{\Pi}\zeta p(\varrho)\big[\zeta\psi(\varrho)\big]_h dxdt
    =\sum\limits_{i=1}^7 \Gamma^{(i)}+I_h,
\end{equation}
where
\begin{equation}\label{alla1}\begin{split}
{\Gamma^{(1)}}=\int_{Q_T} F_1\big[\zeta(\psi(\varrho)-
\varrho\psi'(\varrho))\text{\rm~div~}\mathbf u\big]_h\, dxdt,
\\
{\Gamma^{(3)}}=\int_{Q_T}\Big(g F_1\text{\rm~div~}\mathbf u-
 F_1\nabla u_j \cdot\frac{\partial\mathbf H}{\partial x_j}
- F_2\nabla^\perp u_j \cdot\frac{\partial\mathbf H}{\partial
x_j}\Big)\, dxdt,\\\Gamma^{(2)}=\int_{Q_T} F_1\big[\psi(\varrho)
\nabla\zeta\cdot\mathbf u\big]_h\, dxdt,\quad
\Gamma^{(4)}=-\int_{Q_T}\varrho (\mathbf u\cdot \nabla\zeta)\,
(\mathbf u\cdot \mathbf H)\, dxdt,\\\Gamma^{(5)}=- \int_{Q_T}\zeta
p\nabla\zeta\cdot\mathbf H\, dxdt,\quad \Gamma^{(6)}=
\int_{Q_T}\Big(\mathbb S(\mathbf
    u):\nabla\boldsymbol\xi-
\varrho\mathbf f\cdot\boldsymbol\xi\,\Big) dxdt, \\
\Gamma^{(7)}=\lim\limits_{\tau\to
0}\frac{1}{\tau}\int_{T-\tau}^T\int_\Omega\zeta\varrho\mathbf
u\cdot\mathbf H\,dxdt -\int_\Omega\varrho_0(x)\zeta
\mathbf{u}_0(x)\mathbf H(x,0)\,dx
\end{split}\end{equation}
\begin{equation}\label{lebeg16}
    I_h=-\int_{\mathbb R^2\times[0,T]}\zeta\varrho\mathbf u\cdot\nabla\,
    \text{\rm div}\,\Delta^{-1}\big(\,
    \big[\zeta\psi(\varrho)\big]_h\mathbf u-
    \big[\zeta\psi(\varrho)\mathbf u\big]_h\,\big)\, dxdt.
\end{equation}
Here  $\mathbf F$ is a solution to the Cauchy-Riemann equations
\eqref{moment3}, and $\mathbf H$ is given by \eqref{lebeg11}.
\end{proposition}
 \begin{proof} Recall formulae \eqref{etah} and \eqref{lebeg11} for the cut-off function $\eta_\tau$ and the vector field $\boldsymbol\xi$.
Notice that the function $\eta_\tau\boldsymbol\xi$ and its time
derivative belong to  $L^\infty(0,T; C^k(\Omega))$ and $ L^2(0,T;
C^k(\Omega))$ respectively for all integer $k\geq 0$. Moreover,
$\eta_\tau\boldsymbol\xi$  vanishes at the lateral side and the top
of the cylinder $Q_T$. Hence we can use this function as a test
function in integral identity \eqref{navier8} to obtain
\begin{multline}
\int_{Q_T}\eta_\tau(t)\big(\varrho \mathbf{u}
\cdot\partial_t\boldsymbol{\xi}\big)\, dxdt +
\int_{Q_T}\eta_\tau\big(\varrho \mathbf u\otimes\mathbf
u+p(\varrho)\,\mathbb I- \mathbb S(\mathbf{u})
 \big):\nabla
\boldsymbol{\xi}\, dxdt+\\\int_{Q_T}\eta_\tau\varrho\mathbf{f}\cdot
\boldsymbol{\xi}\, dxdt \label{alla3}=\Gamma_T(\tau),
\end{multline}
where
\begin{equation}\label{alla0}
\Gamma_T(\tau)=\frac{1}{\tau}\int_{T-\tau}^T\int_\Omega\zeta\varrho
\mathbf u\cdot\mathbf H\,dxdt -\int_\Omega\varrho_0(x)\zeta
\mathbf{u}_0(x)\mathbf H(x,0)dx
\end{equation}
Letting $\tau \to 0$ in \eqref{alla3} we arrive at
\begin{multline*}
\int_{Q_T}\big(\varrho \mathbf{u}
\cdot\partial_t\boldsymbol{\xi}\big)\, dxdt + \int_{Q_T}\big(\varrho
\mathbf u\otimes\mathbf u+p(\varrho)\,\mathbb I- \mathbb
S(\mathbf{u})
 \big):\nabla
\boldsymbol{\xi}\, dxdt+ \int_{Q_T}\varrho\mathbf{f}\cdot
\boldsymbol{\xi}\, dxdt=\Gamma^{(7)}.
\end{multline*}
The limit $\Gamma^{(7)}=\lim\limits_{\tau\to 0}\Gamma_T(\tau)$
exists
 since there exists the limit of the  left hand side of  \eqref{alla3}.
 We can rewrite
the latter identity in the form
\begin{multline}\label{alla4}
\int_{Q_T}p(\varrho)\text{div~} \boldsymbol{\xi}\,
dxdt=\Gamma^{(6)}+\Gamma^{(7)}- \int_{Q_T}\varrho \mathbf{u}
\cdot\partial_t\boldsymbol{\xi}\, dxdt - \int_{Q_T}\varrho \mathbf
u\otimes\mathbf u:\nabla \boldsymbol{\xi}\, dxdt.
\end{multline}
It follows from the expression \eqref{lebeg11} for $\boldsymbol\xi$
that $\partial_t\boldsymbol\xi=\zeta\nabla\Delta^{-1}\partial_t g$,
where $g$ is given by \eqref{lebeg7}. From this and and the
representation \eqref{lebeg9} in Lemma \ref{lebeglemma2} we obtain
the identity
\begin{equation}\label{alla5}\begin{split}
    \int_{Q_T}\varrho \mathbf{u}
\cdot\partial_t\boldsymbol{\xi}\, dxdt=\int_{Q_T}\zeta\varrho
\mathbf u\cdot\nabla\Delta^{-1}\big[\zeta(\psi(\varrho)-
\varrho\psi'(\varrho))\text{~div~}\mathbf u\big]_h\, dxdt+\\
\int_{Q_T} \zeta\varrho\mathbf
u\cdot\nabla\Delta^{-1}\big[\psi(\varrho) \nabla\zeta\cdot\mathbf
u\big]_h\, dxdt-\int_{Q_T}\zeta\varrho\mathbf u\cdot\nabla
    \text{\rm div~}\Delta^{-1}\big[\zeta\psi(\varrho)\mathbf u\big]_h\, dxdt
\end{split}\end{equation}
Since $\zeta$,  $\big[\zeta(\psi(\varrho)-
\varrho\psi'(\varrho))\text{~div~}\mathbf u\big]_h$, and
$\big[\psi(\varrho) \nabla\zeta\cdot\mathbf u\big]_h$ are  supported
in $Q_T$, we have
\begin{multline*}
\int_{Q_T}\zeta\varrho \mathbf
u\cdot\nabla\Delta^{-1}\big[\zeta(\psi(\varrho)-
\varrho\psi'(\varrho))\text{~div~}\mathbf u\big]_h\,
dxdt=\int_{\Pi}\zeta\varrho \mathbf
u\cdot\nabla\Delta^{-1}\big[\zeta(\psi(\varrho)-
\varrho\psi'(\varrho))\text{~div~}\mathbf u\big]_h\, dxdt\\
=-\int_{\Pi} F_1\big[\zeta(\psi(\varrho)-
\varrho\psi'(\varrho))\text{~div~}\mathbf u\big]_h\,
dxdt=-{\Gamma^{(1)}},
\end{multline*}
\begin{multline*}
\int_{Q_T} \zeta\varrho\mathbf
u\cdot\nabla\Delta^{-1}\big[\psi(\varrho) \nabla\zeta\cdot\mathbf
u\big]_h\, dxdt= \int_{\Pi} \zeta\varrho\mathbf
u\cdot\nabla\Delta^{-1}\big[\psi(\varrho) \nabla\zeta\cdot\mathbf
u\big]_h\, dxdt=\\
-\int_{\Pi} F_1\big[\psi(\varrho) \nabla\zeta\cdot\mathbf u\big]_h\,
dxdt=-\Gamma^{(2)}.
\end{multline*}
Inserting these equalities into \eqref{alla5} we arrive at
\begin{equation}\label{alla6}\begin{split}
    \int_{\Pi}\varrho \mathbf{u}
\cdot\partial_t\boldsymbol{\xi}\, dxdt=-{\Gamma^{(1)}}-\Gamma^{(2)}
-\int_{Q_T}\zeta\varrho\mathbf u\cdot\nabla
    \text{\rm div~}\Delta^{-1}\big[\zeta\psi(\varrho)\mathbf u\big]_h\,
    dxdt.
\end{split}\end{equation}
Next, expression  \eqref{lebeg11} for $\mathbf H$ implies
\begin{multline}\label{alla7}
  \int_{Q_T}\varrho u_iu_j\frac{\partial \xi_i}{\partial x_j}\, dxdt=
 \int_{\Pi} u_j\frac{\partial\mathbf H}{\partial x_j}\cdot (\zeta\varrho\mathbf u)
 \, dxdt +\\
 \int_{\Pi}\varrho (\mathbf u\cdot \nabla\zeta)\, (\mathbf u\cdot \mathbf H)\, dxdt
=\int_{\Pi} u_j\frac{\partial\mathbf H}{\partial x_j}\cdot
(\zeta\varrho \mathbf u)
 \, dxdt-\Gamma^{(4)}
\end{multline}
Using equation  \eqref{moment3} we obtain the identity
\begin{equation*}
  \int_{\Pi} u_j
  \frac{\partial\mathbf H}{\partial x_j}\cdot (\zeta\varrho\mathbf u)\, dxdt=
  \int_{\Pi}u_j\,(\nabla F_1+\nabla^{\perp} F_2)\cdot \frac{\partial\mathbf H}{\partial x_j}\, dxdt.
\end{equation*}
In view of Corollary \ref{navier12}, the function $\varrho\mathbf u$
belongs to the class $L^2(\Pi)$. Obviously it is supported in
$\Omega\times (0,T)$. From this and formula \eqref{moment4} we
conclude that $\mathbf F\in L^2(0,T; W^{1,2}(B_R))$ for every ball
$B_R\subset \mathbb R^2$. Recall that $\mathbf u\in L^2(0,T;
W^{1,2}(\mathbb R^2))$ is compactly supported in $\Omega\times
(0,T)$. Finally notice that $\mathbf H\in L^\infty(0,T; C^k(\mathbb
R^2))$ for every $k\geq 0$. Hence we can integrate by parts to
obtain
\begin{equation}\label{alla8}\begin{split}
   \int_{\Pi} u_j\frac{\partial\mathbf H}{\partial x_j}
   \cdot (\zeta\varrho \mathbf u)\, dxdt=\int_{\Pi}\Big(F_2 \text{rot~}\frac{\partial\mathbf H}{\partial x_j}- F_1 \text{div~}\frac{\partial\mathbf H}{\partial x_j}
\Big)u_j\, dxdt-\\
\int_{\Pi}\Big( F_1\nabla u_j \cdot\frac{\partial\mathbf H}{\partial
x_j} + F_2\nabla^\perp u_j \cdot\frac{\partial\mathbf H}{\partial
x_j}\Big)\, dxdt.
\end{split}\end{equation}
Recall that $\mathbf H=\nabla\Delta^{-1} g$, where
$g=\big[\zeta\psi(\varrho)\big]_h\in L^\infty (0,T; C^3(\mathbb
R^2))$ is  supported in $Q_T$. It follows that
$$
\text{rot~}\frac{\partial\mathbf H}{\partial x_j}=0, \quad
\text{div~}\frac{\partial\mathbf H}{\partial
x_j}=\partial_{x_j}\big[\zeta\psi(\varrho)\big]_h.
$$
We thus get
\begin{multline*}
\int_{\Pi}\Big(F_2 \text{rot~}\frac{\partial\mathbf H}{\partial
x_j}- F_1 \text{div~}\frac{\partial\mathbf H}{\partial x_j}
\Big)u_j\, dxdt=
-\int_{\Pi} F_1u_j\partial_{ x_j}\big[\zeta\psi(\varrho)\big]_h=\\
-\int_{Q_T} F_1\text{div~}(\big[\zeta\psi(\varrho)\big]_h\mathbf
u)\,dxdt +\int_{Q_T}g F_1\text{div~}\mathbf u\,dxdt.
\end{multline*}
Noting that $F_1=\text{div~}\Delta^{-1}(\zeta\varrho \mathbf u)$ we
arrive at the identity
\begin{multline*}
\int_{\Pi}\Big(F_2 \text{rot~}\frac{\partial\mathbf H}{\partial
x_j}- F_1 \text{div~}\frac{\partial\mathbf H}{\partial x_j}
\Big)u_j\, dxdt=\\
\int_{\Pi}\zeta\varrho \mathbf u\cdot \nabla\text{div~}
\Delta^{-1}(\big[\zeta\psi(\varrho)\big]_h\mathbf
u)\,dxdt+\int_{\Pi}g F_1\text{div~}\mathbf u\,dxdt.
\end{multline*}
Inserting this equality into \eqref{alla8} and recalling the
expression \eqref{alla3} for $\Gamma^{(3)}$ we get
\begin{equation*}\begin{split}
   \int_{\Pi}\zeta\varrho  u_j\frac{\partial\mathbf H}{\partial x_j}\cdot
   \mathbf u\, dxdt=-\Gamma^{(3)}-\int_{\Pi}\zeta\varrho \mathbf u\cdot \nabla\text{div~}
\Delta^{-1}(\big[\zeta\psi(\varrho)\big]_h\mathbf u)\,dxdt.
\end{split}\end{equation*}
Inserting this result into \eqref{alla7} we finally obtain
\begin{equation}\label{alla10}\begin{split}
\int_{\Pi}\varrho u_iu_j\frac{\partial \xi_i}{\partial x_j}\, dxdt=
-\Gamma^{(3)}-\Gamma^{(4)}+\int_{\Pi}\zeta\varrho \mathbf u\cdot
\nabla\text{div~} \Delta^{-1}(\big[\zeta\psi(\varrho)\big]_h\mathbf
u)\,dxdt.
\end{split}\end{equation}
It remains to note that in view of  \eqref{alla1} and
\eqref{lebeg11} we have
\begin{equation}\label{alla11}\begin{split}
\int_{\Pi}p\,\text{\rm div~}\boldsymbol\xi
    dxdt=
    \int_{\Pi}\zeta p\big[\zeta\psi(\varrho\big]_h\, dxdt
-\Gamma^{(5)} .
\end{split}\end{equation}
Inserting \eqref{alla6}, \eqref{alla10}, and \eqref{alla11} into
\eqref{alla4} we obtain the desired identity \eqref{lebeg15}.
\end{proof}

\subsection{Step 3. Estimates of $\Gamma^{(i)}$ }
In this section we show that the quantities  $\Gamma^{(i)}$  in the
basic integral identity \eqref{lebeg15} are bounded by a constant,
depending only on the exponent $\lambda$, the cut-off function
$\zeta$, and the constant $E$ specified by Remark
\ref{stokesremark}.
 \begin{proposition}\label{anna1} Under the assumptions of Theorem \ref{lebegtheorem},
\begin{equation}\label{lebeg19}
\Gamma^{(i)}\leq c(\zeta, \lambda)E,
\end{equation}
where $c$ depends only on $\zeta$ and $\lambda$.
\end{proposition}
\begin{proof} Let us estimate $\Gamma^{(1)}$ and $\Gamma^{(2)}$.
Since $\lambda<1/6$, we can choose $\nu>0$, depending on $\lambda$,
such that $(3-\nu)^{-1}+(1+2\lambda)2^{-1}=1$.  Lemmas
\ref{momentlemma2}, \ref{lebeglemma1} and the H\"{o}lder inequality
imply
\begin{equation}\label{lebeg21}\begin{split}
  | \Gamma^{(1)}|\leq \|F_1\|_{L^2(0,T; L^{3-\nu}(\mathbb R^2))}\, \| \big[\zeta(\psi(\varrho)-
  \varrho\psi'(\varrho))\text{~div~}\mathbf u\big]_h\|_{L^2(0,T; L^{2/(1+2\lambda)}(\mathbb R^2))}\leq\\
\|F_1\|_{L^2(0,T; L^{3-\nu}(\mathbb R^2))}\, \| \zeta(\psi(\varrho)-
  \varrho\psi'(\varrho))\text{~div~}\mathbf u\|_{L^2(0,T; L^{2/(1+2\lambda)}(\mathbb R^2))} \leq c(\lambda, \zeta) E.
\end{split}\end{equation}
Applying   Lemmas \ref{momentlemma2} and \ref{lebeglemma1} once more
we obtain
\begin{equation}\label{lebeg22}\begin{split}
  | \Gamma^{(2)}|\leq \|F_1\|_{L^2(0,T; L^{3-\nu}(\mathbb R^2))}\, \| \big[\psi(\varrho)\nabla\zeta\cdot\mathbf u\big]_h\|_{L^2(0,T; L^{2}(\mathbb R^2))}\leq\\
\|F_1\|_{L^2(0,T; L^{3-\nu}(\mathbb R^2))}\,
\|\psi(\varrho)\nabla\zeta\cdot\mathbf u\|_{L^2(0,T; L^{2}(\mathbb
R^2))} \leq c(\lambda, \zeta) E.
\end{split}\end{equation}
 Our next task is to estimate $\Gamma^{(3)}$ and $\Gamma^{(4)}$.
 Since $\mathbf u$ is
supported in $\Omega\times[0,T]$, we have
\begin{equation}\label{lebeg28}
|\Gamma^{(3)}|\leq \int_{\Omega\times (0,T)}|\mathbf F|\, |\nabla
\mathbf u|(|g|+|\nabla \mathbf H|)\, dxdt
\end{equation}
It follows from Lemmas \ref{lebeglemma2} and \ref{lebeglemma3} that
\begin{equation}\label{lebeg29}
\|g\|_{L^\infty(0,T; L^{1/\lambda}(\mathbb R^2))}+\|\nabla \mathbf
H\|_{L^\infty(0,T; L^{1/\lambda}(\mathbb R^2))}\leq
c(\zeta,\lambda)E.
\end{equation}
On the other hand, energy estimate \eqref{navier7} and Lemma
\ref{momentlemma2} imply
\begin{equation}\label{lebeg30}
\|\mathbf F\|_{L^2(0,T; L^{3-\kappa}(\Omega))}\leq c(\zeta,
\kappa)E, \quad \|\nabla \mathbf u\|_{L^2(0,T; L^{2}(\Omega))}\leq
c(\zeta,\lambda)E,
\end{equation}
where $\kappa$ is an arbitrary positive number. Since $\lambda<1/6$,
we can choose $\kappa$ such that $(3-\kappa)^{-1}+\lambda+2^{-1}=1$.
Applying the H\"{o}lder inequality and using \eqref{lebeg29},
\eqref{lebeg30} we obtain
\begin{multline*}
\int_{\Omega\times (0,T)}|\mathbf F|\, |\nabla \mathbf u|(|g|+
|\nabla \mathbf H|)\, dxdt\leq\\
\|\mathbf F\|_{L^2(0,T; L^{3-\kappa}(\Omega))}\,\|\nabla \mathbf
u\|_{L^2(0,T; L^{2}(\Omega))} (\|g\|_{L^\infty(0,T;
L^{1/\lambda}(\mathbb R^2))}+\|\nabla \mathbf H\|_{L^\infty(0,T;
L^{1/\lambda}(\mathbb R^2))})\leq c(\zeta,\lambda)E.
\end{multline*}
which leads to estimate \eqref{lebeg19} for
   $\Gamma^{(3)}$.
In order to estimate $\Gamma^{(4)}$ notice that in view of Lemma
\ref{lebeglemma3} and energy estimate \eqref{navier7}, we have
\begin{multline*}
|\Gamma^{(4)}| \leq c(\zeta)\int_{\Omega\times (0,T)} |\mathbf H|\varrho |\mathbf u|^2\,dxdt \leq\\
 \leq c(\zeta)\|\mathbf H\|_{L^\infty(\Omega\times (0,T))}\int_{\Omega\times (0,T)} \varrho |\mathbf u|^2\,dxdt \leq c(\zeta,\lambda)E.
\end{multline*}
 Next  we employ
Lemma \ref{lebeglemma3} and estimate \eqref{navier7} to obtain
\begin{multline*}
|\Gamma^{(5)}|\leq \int_{\mathbb R^2\times (0,T)}\zeta
p|\nabla\zeta|\, |\mathbf H|\, dxdt
 \leq c(\zeta)\|\mathbf H\|_{L^\infty(\Omega\times (0,T))}\int_{\Omega\times (0,T)}
  p\,dxdt \leq c(\zeta,\lambda)E.
\end{multline*}
It remains to estimate  $\Gamma^{(6)}$ and $\Gamma^{(7)}$.
Expression  \eqref{lebeg11} for the vector field $\boldsymbol\xi$,
and expression \eqref{alla1} for $\Gamma^{(6)}$ yield the estimate
\begin{multline}\label{lebeg34}
|\Gamma^{(6)}|\leq \Big|\int_{\Pi}\mathbb S(\mathbf
    u):\nabla\boldsymbol\xi\,dxdt\Big|+\Big|\int_{\Pi}
\varrho\mathbf f\cdot\boldsymbol\xi\, dxdt\Big| \leq\\
\int_{\Pi}\zeta|\mathbb S(\mathbf
    u)|\, |H|\,dxdt+\int_{\Pi}(\zeta+|\nabla\zeta|)|\mathbf H|(|\mathbb S(\mathbf
    u)|+\varrho |\mathbf f|)\, dxdt\leq\\
    c(\zeta)\int_{Q_T}|\nabla\mathbf u|(|\mathbf
    H|+|\nabla\mathbf H|)\, dxdt +c(\zeta)E\int_{Q_T}\varrho|\mathbf H|\, dxdt.
\end{multline}
On the other hand, energy estimate \eqref{navier7} yields
\begin{multline*}
\int_{Q_T}|\nabla\mathbf u|(|\mathbf
    H|+|\nabla\mathbf H|)\, dxdt \leq c(\zeta)
    \|\nabla\mathbf u\|_{L^2(Q_T)}
    \|\mathbf H\|_{L^2(0,T; W^{1,2}(\Omega))}
\leq\\ c(\zeta)E\|\mathbf H\|_{L^2(0,T; W^{1,2}(\Omega))}
\end{multline*}
Next,  Lemma \ref{lebeglemma3} and the inequality $1/\lambda>2$
imply
\begin{equation*}
 \|\mathbf H\|_{L^2(0,T; W^{1,2}(\Omega))}\leq
  \|\mathbf H\|_{L^\infty(0,T; L^\infty(\Omega))}+
\|\nabla \mathbf H\|_{L^\infty(0,T; L^{1/\lambda}(\Omega))}\leq
c(\zeta,\lambda) E.
\end{equation*}
We thus get
\begin{equation}\label{lebeg35}
\int_{Q_T}|\nabla\mathbf u|(|\mathbf
    H|+|\nabla\mathbf H|)\, dxdt \leq c(\zeta,\lambda)E
\end{equation}
Finally notice that
\begin{equation*}
\int_{Q_T}\varrho|\mathbf H|\, dxdt\leq \|\mathbf
    H\|_{L^\infty(Q_T)}\,\|\varrho\|_{L^\infty(0,T; L^1(\Omega))}
\leq c(\zeta, \lambda)E.
\end{equation*}
Inserting  this estimate along with \eqref{lebeg35} into
\eqref{lebeg34} we arrive at the desired estimate  \eqref{lebeg19}
for   $\Gamma^{(6)}$. Finally,  expression \eqref{alla0} for
$\Gamma_T$, Lemma
 \ref{lebeglemma3}, and the energy estimate
\eqref{navier7} imply
\begin{equation}\label{lebeg37}
|\Gamma_T(\tau)|\leq c(\zeta)\|\mathbf
    H\|_{L^\infty(Q_T)}\|\varrho\|_{L^\infty(0,T;
    L^1(\Omega))}\leq c(\zeta)E.
\end{equation}
It follows from this that $\Gamma^{(7)}=\lim\limits_{\tau\to
0}\Gamma_T(\tau)$ satisfies inequality \eqref{lebeg19}.
\end{proof}

\subsection{Step 4. Proof of Theorem \ref{lebegtheorem}}
The proof is based on  Propositions \ref{lebegproposition} and
\ref{anna1}. First we show that the quantity $I_h$ in  identity
\eqref{lebeg15} tends to zero as $h\to 0$. We begin  with the
observation that
\begin{equation*}\begin{split}
    I_h=\int_{\Pi}\zeta\varrho\mathbf u
    \cdot\nabla\,\text{\rm div}\,
    \Delta^{-1}\Big(\,\big[\zeta\psi(\varrho)\mathbf u\big]_h-
    \zeta\psi(\varrho)\mathbf
    u\Big)\, dxdt-\\\int_{\Pi}\zeta\varrho\mathbf u
    \cdot\nabla\,\text{\rm div}\,\Delta^{-1}\Big(\,\big(
    \big[\zeta\psi(\varrho)\big]_h- \zeta\psi(\varrho)\big)\mathbf
    u\Big)\, dxdt.
\end{split}\end{equation*}
Since the Riesz operator $\nabla\,\text{\rm div}\,\Delta^{-1}:
L^2(\mathbb R^2)\to L^2(\mathbb R^2)$ is bounded and $\zeta$ is
supported in $\Omega$, we have
\begin{equation}\label{lebeg42}
    |I_h|^2\leq c\|\varrho\mathbf u\|_{L^2(Q_T)}^2
    \int_0^T (J_h(t)+L_h(t))\, dt ,
\end{equation}
where
$$ J_h(t)=
\|\big[\zeta\psi(\varrho)\mathbf u\big]_h(t)-
\zeta\psi(\varrho)\mathbf u(t)\|_{ L^{2}(\Omega)}^2, \,\,
L_h(t)=\|\big[\zeta\psi(\varrho)\big]_h\mathbf
u(t)-\zeta\psi(\varrho)\mathbf u(t) \|_{ L^{2}(\Omega)}^2.
$$
In view of Corollary \ref{navier13} the vector field $\varrho\mathbf
u$ belongs to the class $L^2(Q_T)$. Hence it suffices to
 prove that  the sequence $J_h+L_h$ has an integrable majorant and
tends to $0$ a.e. in $(0,T)$ as $h\to 0$.  Inequality \eqref{lebeg4}
in Lemma \ref{lebeglemma1} implies that $\zeta\psi(\varrho)\mathbf
u\in L^2(0,T; L^2(\mathbb R^2))$. It follows from  properties of the
mollifier that $
 J_h(t)\to 0$ as $h\to
 0$ a.e. in $(0,T)$.
On the other hand, we have
$$
J_h(t)\leq
 \|\big[\psi(\varrho)\mathbf u\big]_h(t)\|_{ L^{2}(\Omega)}+
 \|\psi(\varrho)\mathbf u(t)\|_{ L^{2}(\Omega)}\leq
  2\|\psi(\varrho)\mathbf u(t)\|_{ L^{2}(\Omega)}
 $$
By \eqref{lebeg4}, the right hand side is integrable  over $[0,T]$
and is independent of $h$. Hence the sequence  $J_h\to 0$ a.e. in
$(0,T)$ and has an integrable majorant. Next, the H\"{o}lder
inequality implies
$$
L_h(t)\leq \|\big[\zeta\psi(\varrho)\big]_h(t)-\zeta\psi(\varrho)(t)
\|_{ L^{1/\lambda}(\Omega)}^2\|\mathbf
u(t)\|_{L^{2/(1-2\lambda)}(\Omega)}^2
$$
Since the embedding $W^{1,2}_0(\Omega)\hookrightarrow
L^{2/(1-2\lambda)}(\Omega)$ is bounded, we have
\begin{equation}\label{lebeg52}
L_h(t)\leq
c(\lambda)\|\big[\zeta\psi(\varrho)\big]_h(t)-\zeta\psi(\varrho)(t)
\|_{ L^{1/\lambda}(\Omega)}^2\|\mathbf u(t)\|_{W^{1,2}_0(\Omega)}^2.
\end{equation}
Since $\zeta\psi(\varrho)\in L^\infty(0,T; L^{1/\lambda}(\Omega))$,
we have
$$
\|\big[\zeta\psi(\varrho)\big]_h(t)-\zeta\psi(\varrho)(t) \|_{
L^{1/\lambda}(\Omega)}^2\to 0\text{~~as~~} h\to 0\text{~~for a.e.
~~}t\in (0,T).
$$
Hence $L_h(t)\to 0$ a.e. in $(0,T)$. Notice that
\begin{multline*}
\|\big[\zeta\psi(\varrho)\big]_h(t)-\zeta\psi(\varrho)(t) \|_{
L^{1/\lambda}(\Omega)}\leq \\\|\big[\zeta\psi(\varrho)\big]_h(t)
\|_{ L^{1/\lambda}(\Omega)}+\|\zeta\psi(\varrho)\big(t) \|_{
L^{1/\lambda}(\Omega)}\leq 2 \|\zeta\psi(\varrho)\big(t) \|_{
L^{1/\lambda}(\Omega)}.
\end{multline*}
Combing this result with \eqref{lebeg52} and recalling that
$\zeta\psi(\varrho)\in L^\infty(0,T; L^{1/\lambda}(\Omega))$ we
obtain $ L_h(t)\leq c\|\mathbf u(t)\|_{W^{1,2}_0(\Omega)}^2. $ In
view of the  energy estimate \eqref{navier7} the right side of this
inequality is integrable over $(0,T)$. Hence the sequence $L_h$ has
an integrable majorant. Applying the Lebesgue dominant convergence
Theorem we arrive at the relation
\begin{equation*}
    \int_0^T (J_h(t)+L_h(t))\, dt \to 0 \text{~~as~~} h\to 0.
\end{equation*}
From this and \eqref{lebeg42} we conclude that
 $I_h\to 0$ as $h\to 0$. Next,  Propositions \ref{lebegproposition} and
\ref{anna1} imply
\begin{equation}\label{lebeg43}
    \int_{Q_T} \zeta p\big[\zeta\psi(\varrho)\big]_h\, dxdt
    \leq I_h +c(\zeta, \lambda)E.
\end{equation}
The functions $\big[\zeta\psi(\varrho)\big]_h$ are nonnegative and
converge a.e. in $Q_T$ to $\zeta\psi(\varrho)$. Letting $h\to 0$ in
\eqref{lebeg43} and applying the Fatou  Theorem we arrive at the
inequality
\begin{equation*}
    \int_{Q_T} \zeta^2 p\psi(\varrho)\, dxdt
    \leq c(\zeta, \lambda)E.
\end{equation*}
It remains to note that $\psi(\varrho)\geq c\varrho^\lambda-1$ and
the theorem follows.

\section{Proof of Theorems \ref{stokes14} and
\ref{stokes16}}\label{final}
\subsection{Proof of Theorem \ref{stokes14}}
 By Proposition \ref{navierproposition}, for
every $\varepsilon>0$ regularized problem \eqref{navier1} has a
solution $( \varrho_\epsilon, \mathbf u_\epsilon)$ which admits
estimates \eqref{navier7} and satisfies integral identities
\eqref{navier8}, \eqref{navier9}.

\begin{lemma}\label{finallemma1}
Let $\lambda\in [0,1/6)$ and $\Omega'\Subset \Omega$. Then there are
exponent $r, p\in (2,\infty)$ and $q,s, z\in (1,\infty)$ such that
\begin{gather}\label{final1}
\|\varrho_\varepsilon\|_{L^{1+\lambda}(\Omega'\times
(0,T))}+\varepsilon\int_{\Omega'\times(0,T)}
{\varrho_\varepsilon}^{\gamma+\lambda}\leq C,\\\label{final2}
\|\varrho_\varepsilon\|_{L^{r}(0,T;L^s(\Omega'))}+
\|\varrho_\varepsilon\mathbf
u_\varepsilon\|_{L^{p}(0,T;L^z(\Omega'))}+\|\varrho_\varepsilon|\mathbf
u_\varepsilon|^2\|_{L^{q}(\Omega'\times (0,T))}\leq C,
\end{gather}
where $C$ is independent of $\varepsilon$. Moreover, the sequences
$\varrho_\varepsilon\mathbf u_\varepsilon$ and $\varrho_\varepsilon$
are equi-integrable.
\end{lemma}
\begin{proof} Fix an arbitrary $\Omega'\Subset \Omega$. Choose a nonnegative
function $\zeta\in C^\infty_0(\mathbb R^2)$ with the properties:
 $\zeta$ is compactly supported in $\Omega$ and $\zeta=1$ in
$\Omega'$. Notice  that $\lambda$, $\zeta$, and $\varrho_\varepsilon
$ meet all requirements of Theorem \ref{lebegtheorem}. Hence
$p(\varrho_\varepsilon)$ satisfy  inequality \eqref{lebeg1}.
 It is easy to see that estimates \eqref{final1} follows from \eqref{lebeg1} and the formula
 $p(\varrho)=\varrho_\varepsilon+\varrho_\varepsilon^\gamma$. Next choose an arbitrary $r\in (2,\infty)$
 and set $s=r/(r-\lambda)>1$ and $\alpha=(1+\lambda)/r\in (0,1)$.  Obviously
$$
(1-\alpha)/\infty+\alpha/(1+\lambda)=1/r, \quad 1-\alpha
+\alpha/(1+\lambda)=1/s.
$$
From this, inequality \eqref{final1}, and the interpolation
inequality we obtain
\begin{equation}\label{final3}
\|\varrho_\varepsilon\|_{L^{r}(0,T;L^s(\Omega'))}\leq
\|\varrho_\varepsilon\|_{L^{\infty}(0,T;L^1(\Omega'))}^{1-\alpha}\|\varrho_\varepsilon\|_{L^{1+\lambda}
(0,T;L^{1+\lambda}(\Omega'))}^\alpha<C,
\end{equation}
which gives the estimate \eqref{final2} for $\varrho_\varepsilon$.
In order to estimate the quantity $\varrho_\varepsilon|\mathbf
u_\varepsilon|$, represent it in the form
\begin{equation}\label{final4}
 \varrho_\varepsilon|\mathbf u_\varepsilon|=\varrho_\varepsilon^\mu  (\varrho_\varepsilon|\mathbf u_\varepsilon|^2)^\beta |\mathbf u_\varepsilon|^\nu.
\end{equation}
Let us show that there exist exponents $\mu\in (1/2,1)$, $\beta,
\nu\in (0,1)$ and $p,z,\sigma\in (1, \infty)$ with the properties
\begin{equation}\label{final5}\begin{split}
     \beta=1-\mu, \quad \nu+2\beta=1, \quad \text{i.e.,~~} \nu=2\mu-1,\\
\mu/r+\nu/2=1/p, \quad \mu/s+\beta+\nu/\sigma=1/z.
\end{split}\end{equation}
To this end notice that these relations can be equivalently
rewritten in the form
\begin{equation*}\begin{split}
1/p=(2\mu-1)/2+\mu/r, \,\,1/z=1+\mu(1/s-1)+(2\mu-1)/\sigma, \,\,
\beta=1-\mu,\,\,  \nu=2\mu-1,
\end{split}\end{equation*}
which gives $\mu=(1/2+1/p)(1+1/r)^{-1}$. The inequalities
$1/2<\mu<1$ are fulfilled if and only if $2r/(r+2)<p<2r$. Since
$r>2$, there exists $p>2$ satisfying these inequalities. On the
other hand, it follows from $s>1$ that $0<1+\mu(1/s-1)<1$. Hence
there is $\sigma\in (1,\infty)$ such that $z\in (1,\infty)$. This
completes the proof of the existence of  exponents satisfying
\eqref{final5}. The H\"{o}lder inequality,  estimate \eqref{final1},
and energy estimate \eqref{navier7} imply
\begin{equation*}\begin{split}
    \|\varrho_\varepsilon\mathbf u_\varepsilon\|_{L^p(0,T;L^z(\Omega'))}\leq \\
\|\varrho_\varepsilon^\mu\|_{L^{r/\mu}(0,T;L^{s/\mu}(\Omega'))}
\|\varrho_\varepsilon^\beta|\mathbf
u_\varepsilon|^{2\beta}\|_{L^\infty(0,T;L^{1/\beta}(\Omega'))}
\||\mathbf
u_\varepsilon|^\nu\|_{L^{2/\nu}(0,T;L^{\sigma/\nu}(\Omega'))}\\=
\|\varrho_\varepsilon\|_{L^{r}(0,T;L^{s}(\Omega'))}^\mu
\|\varrho_\varepsilon|\mathbf
u_\varepsilon|^{2}\|_{L^\infty(0,T;L^{1}(\Omega'))}^\beta \|\mathbf
u_\varepsilon\|_{L^{2}(0,T;L^{\sigma}(\Omega'))}^\nu\leq C \|\mathbf
u_\varepsilon\|_{L^{2}(0,T;L^{\sigma}(\Omega'))}^\nu.
\end{split}\end{equation*}
Recall that the embedding $W^{1,2}_0(\Omega)\hookrightarrow
L^\sigma(\Omega)$ is bounded for every $\sigma\in [1, \infty)$. It
follows from this and energy estimate \eqref{navier7} that
$$
\|\mathbf u_\varepsilon\|_{L^{2}(0,T;L^{\sigma}(\Omega'))}\leq
c\|\mathbf u_\varepsilon\|_{L^{2}(0,T;W^{1,2}_0(\Omega))}\leq C,
$$
which leads to the estimate for $\varrho_\varepsilon\mathbf
u_\varepsilon$ in \eqref{final2}. Now our task is  to estimate the
kinetic energy density $\varrho_\varepsilon|\mathbf
u_\varepsilon|^2$ Since $p>2$ and $z>1$ there are $\kappa_1,
\kappa_2, \omega\in (1,\infty)$ such that $1/p+1/2= 1/\kappa_1$ and
$1/z+1/\omega=1/\kappa_2$. Set $q=\min\{\kappa_i\}$. Applying the
H\"{o}lder inequality and using estimate \eqref{final2} for
$\varrho_\varepsilon\mathbf u_\varepsilon$ we obtain
\begin{equation*}\begin{split}
    \|\varrho_\varepsilon|\mathbf u_\varepsilon|^2\|_{L^q(\Omega'\times(0,T))}\leq
    C \|\varrho_\varepsilon|\mathbf u_\varepsilon|^2\|_{L^{\kappa_1}(0,T; L^{\kappa_2}(\Omega'))}\\
\leq c\|\varrho_\varepsilon \mathbf
u_\varepsilon|\|_{L^p(0,T;L^z(\Omega'))} \|\mathbf
u_\varepsilon\|_{L^{2}(0,T;L^{\omega}(\Omega'))}\leq C \|\mathbf
u_\varepsilon\|_{L^{2}(0,T;W^{1,2}_0(\Omega))}\leq C.
\end{split}\end{equation*}
This completes the proof of \eqref{final2}. It remains to show that
the sequences $\varrho_\varepsilon$ and $\varrho_\varepsilon\mathbf
u_\varepsilon$ are equi-integrable in $Q_T$. By energy estimate
\eqref{navier7}, the sequence $\varrho_\varepsilon\log(1+
\varrho_\varepsilon)$ is bounded in $L^1(Q_T)$.  Hence this sequence
is equi-integrable. This means that for every $\varkappa>0$ there is
$\delta(\varkappa)$, depending on $\epsilon $ only, such that the
inequality
$$
\int_A\varrho_\varepsilon\, dxdt\leq \varkappa
$$
hold for every $A\subset Q_T$ such that $\text{\rm
meas~}A<\delta(\varkappa)$.  By the Cauchy inequality and energy
estimate \eqref{navier7}, we have
$$
\int_A\varrho_\varepsilon|\mathbf u_\epsilon|\, dxdt\leq
\Big(\int_A\varrho_\varepsilon\, dxdt\Big)^{1/2}\Big(\int_{Q_T}
\varrho_\varepsilon|\mathbf u_\varepsilon|^2\,dxdt \Big)^{1/2}\leq E
\sqrt{\varkappa}.
$$
which yields the equi-integrability of the sequence
 $\varrho_\varepsilon\mathbf u_\epsilon$.
\end{proof}

Let us turn to the proof of Theorem \ref{stokes14}. It follows from
energy estimates \eqref{navier7}  and Lemma \ref{finallemma1} that,
after passing to a subsequence if necessary, we can assume that
there are integrable functions $\varrho$, $\mathbf u$,
$\overline{\varrho\mathbf u}$, and $\overline{\varrho \mathbf
u\otimes\mathbf u}$ with the properties
\begin{equation}\label{final6}\begin{split}
\varrho_\varepsilon \rightharpoonup \varrho,\quad
\varrho_\varepsilon\mathbf u_\varepsilon\rightharpoonup
\overline{\varrho\mathbf u}\quad &\text{weakly in~}
L^1(\Omega\times (0,T)),\\
\mathbf{u}_\varepsilon \rightharpoonup \mathbf{u}\quad &\text{weakly
in~} L^2(0,T; W^{1,2}_0(\Omega)),
\end{split}\end{equation}
For every compact set $\Omega'\subset\Omega$, we have
\begin{equation}\label{final7}\begin{split}
\varrho_\varepsilon \rightharpoonup \varrho\text{~weakly in~}
L^r(0,T; L^s(\Omega')),\,\,\varrho_\varepsilon
\mathbf{u}_\varepsilon \rightharpoonup \overline{\varrho\mathbf{u}}
\text{~weakly in~} L^p(0,T;
L^{z}(\Omega')),\\
\varrho_\varepsilon\mathbf u_\varepsilon\otimes
 \mathbf u_\varepsilon\rightharpoonup
\overline{\varrho\mathbf u\otimes\mathbf u}\text{~weakly
in~}L^{q}(\Omega'\times (0,T)).
\end{split}\end{equation}
Here $r,p\in (2,\infty)$ and $q,s, z\in (1,\infty)$ are given by
Lemma \ref{finallemma1}. It follows from energy estimate
\eqref{navier7} and convexity of the function
$\varrho\log(1+\varrho)$  that $\varrho$ and $\mathbf u$ satisfies
inequalities \eqref{stokes12}. Moreover, $\varrho\in L^r(0,T;
L^s(\Omega'))$, $\overline{\varrho\mathbf{u}}\in L^p(0,T;
L^{z}(\Omega'))$ and $\overline{\varrho\mathbf u\otimes\mathbf u}\in
L^{q}(\Omega'\times (0,T))$ for ever $\Omega'\Subset\Omega$. Finally
notice that in view of estimates \eqref{final1} we have
$$
\varepsilon\int_{\Omega'\times(0,T)} {\varrho_\varepsilon}^{\gamma}
\to 0\text{~~as~~} \varepsilon\to 0.
$$
Substituting $(\varrho_\varepsilon, \mathbf u_\varepsilon)$ and
$\varphi(\varrho_\varepsilon):=\varrho_\varepsilon$ into
\eqref{navier8},  \eqref{navier10} and letting $\varepsilon\to 0$ we
obtain that the integral identities
\begin{multline}\label{final8}
\int_{Q_T}\big (\overline{\varrho\mathbf{u}}\cdot
\partial_t\boldsymbol{\xi}+
\overline{\varrho\mathbf{u}\otimes\mathbf{u}}:\nabla\boldsymbol{\xi}+\varrho\div
\boldsymbol{\xi}-\mathbb
S(\mathbf{u}):\nabla\boldsymbol{\xi}\big)\,dxdt \\+\int_{Q}
\varrho\mathbf f\cdot\boldsymbol{\xi}\, dxdt+ \int_{\Omega }(
\varrho_0\mathbf u_0\cdot \boldsymbol{\xi})(x,0)\,dx=0
\end{multline}
\begin{equation}\label{final9}
\int_{Q_T}\big (\varrho
\partial_t\psi+\overline{\varrho\mathbf{u}}\cdot\nabla\psi
\big)\,dxdt+\int_{\Omega } \varrho_0(x)\psi(x,0)\, dx=0
\end{equation}
hold  for all vector fields $\boldsymbol{\xi}\in C^\infty(Q_T)$
equal
  $0$ in a neighborhood of $\partial\Omega\times[0,T]$ and of the top
$\Omega\times\{t=T\}$ and for all  $\psi\in C^\infty(Q_T)$ vanishing
in a neighborhood of the top $\Omega\times\{t=T\}$. It remains to
prove that
\begin{equation}\label{final10}
\overline{\varrho\mathbf{u}}=\varrho\mathbf{u},\quad
\overline{\varrho\mathbf{u}\otimes\mathbf{u}}=\varrho\mathbf{u}\otimes\mathbf{u}
\text{~~a.e. in~~} Q_T
\end{equation}
The proof is standard, see \cite{FEIRBOOK}. We begin with the
observation that $\varrho_\varepsilon$ and $\mathbf u_\varepsilon$
satisfies the equations
\begin{equation}\label{final11}\begin{split}
    \partial_t(\varrho_\varepsilon\mathbf u_\varepsilon)=\text{~div~}
    \big(\mathbf S(\mathbf
    u_\varepsilon)-\varrho_\varepsilon\mathbf u_\varepsilon\otimes
 \mathbf u_\varepsilon\big)-\nabla
 p(\varrho_\varepsilon)+\varrho_\varepsilon \mathbf f,\,\,
\partial_t\varrho_\varepsilon=-\text{~div~}\big(\,\varrho_\varepsilon\mathbf
u_\varepsilon\,\big),
\end{split}
\end{equation}
which are understood in the sense of the distribution theory.Notice
that in view of the energy estimate \eqref{navier7}, the sequences
$\varrho_\varepsilon$, $\mathbf S(\mathbf
    u_\varepsilon)$, $\varrho_\varepsilon\mathbf u_\varepsilon\otimes
 \mathbf u_\varepsilon$, $p(\varrho_\varepsilon)$ are bounded in the
 space $L^2(0,T; L^1(\Omega))$.
 Choose an arbitrary function $\xi\in C^\infty_0(Q_T)$.
 Since the embedding
 $W^{3,2}_0(\Omega)\to C^1_0(\Omega)$ is bounded, it follows from
 \eqref{final11} that the sequences $\partial_t(\xi
 \varrho_\varepsilon)$ and $
 \partial_t(\xi\varrho_\varepsilon\mathbf u_\varepsilon)$ are
 bounded in $L^2(0,T; W^{-3,2}(\Omega))$. On the other hand, Lemma
 \ref{finallemma1} implies that the sequences
 $\xi
 \varrho_\varepsilon$ and $
 \xi\varrho_\varepsilon\mathbf u_\varepsilon$ are bounded in
 $L^r(0,T; L^s(\Omega))$ and $L^p(0,T; L^z(\Omega))$ respectively . Notice that
 $r, p>2$ and the embedding $W^{-1,2}(\Omega)\hookrightarrow
L^s(\Omega)$, $W^{-1}(\Omega)\hookrightarrow L^z(\Omega)$ is compact
for $s, z>1$. Applying the Dubinskii-Lions compactness Theorem we
conclude that the sequences $\xi
 \varrho_\varepsilon$ and $
 \xi\varrho_\varepsilon\mathbf u_\varepsilon$ are relatively
 compact in $L^2(0,T; W^{-1,2}(\Omega))$.  From this and \eqref{final6}
 we obtain
$$
\int_{Q_T}\xi\varrho_\varepsilon\mathbf u_\varepsilon\,dxdt \to
\int_{Q_T}\xi\varrho\mathbf u\,dxdt,\,\,\,
\int_{Q_T}\xi\varrho_\varepsilon\mathbf u_\varepsilon\otimes\mathbf
u_\varepsilon\,dxdt \to \int_{Q_T}\xi\varrho\mathbf u\otimes\mathbf
u\,dxdt
$$
as $\varepsilon\to 0$, which yields \eqref{final10}. This completes
the proof of Theorem \ref{stokes14}.
\subsection{Proof of Theorem \ref{stokes16}}
It suffices to note that estimate \eqref{stokes17} follows directly
from Theorem \ref{radontheorem}, and estimates \eqref{19} follow
from Proposition \ref{momentproposition} and Theorem
\ref{lebegtheorem}.
\begin{appendix}
\section{Proof of Lemmas \ref{sobolevlemma1} and \ref{sobolevlemma2}}\label{applem1}
\paragraph{Proof of Lemma \ref{sobolevlemma1}}Introduce the polar coordinates $\lambda=|\xi|\in \mathbb R^+$ and
$\boldsymbol\omega=|\xi|^{-1}\xi\in \mathbb S^1$. Applying the
Fubini Theorem we obtain
\begin{equation*}\begin{split}
    \mathfrak F g(\xi)=\frac{1}{2\pi}\int_{\mathbb R^2} e^{-ix\cdot \xi}g(x)\,dx=
   \frac{1}{2\pi}\int_{\mathbb R^2} e^{-i \lambda
   \boldsymbol \omega\cdot x}g(x)\, dx=\\\frac{1}{2\pi}\int_{-\infty}^\infty e^{-i\lambda \tau}\Big\{
   \int\limits_{\boldsymbol\omega\cdot x=\tau} g(x)\, dl\Big\}\,d\tau=
   \frac{1}{2\pi}\int_{-\infty}^\infty e^{-i\lambda \tau}
   \Phi(\boldsymbol\omega, \tau)\, d\tau=
   \frac{1}{\sqrt{2\pi}}\mathfrak F_\lambda \Phi(\boldsymbol\omega, \lambda),
\end{split}\end{equation*}
where $\mathfrak F_\lambda$ is the Fourier transform with respect to
$\tau$. We thus get
\begin{equation}\label{sobolev7}
| \mathfrak F g(\lambda \omega)|^2=\frac{1}{2\pi}|\mathfrak
F_\lambda \Phi(\boldsymbol\omega, \lambda)|^2
\end{equation}
 Since $\Phi$ is a real valued function, the
Plancherel identity yields
$$
\int_0^\infty |\mathfrak F_\lambda\Phi(\boldsymbol\omega,
\lambda)|^2d\lambda=\frac{1}{2} \int_{-\infty}^\infty |\mathfrak
F_\lambda \Phi(\boldsymbol\omega, \lambda)|^2d\lambda=\frac{1}{2}
\int_{-\infty}^\infty | \Phi(\boldsymbol\omega, \tau)|^2d\tau.
$$
Integrating both sides of \eqref{sobolev7} by $\lambda$ we conclude
that
$$
\int_{0}^\infty |\mathfrak F g(\lambda
\boldsymbol\omega)|^2d\lambda=\frac{1}{4\pi} \int_{-\infty}^\infty |
\Phi(\boldsymbol\omega, \tau)|^2d\tau
$$
It follows that
\begin{multline*}
\int\limits_{\mathbb R^2}|\xi|^{-1}|\mathfrak F g|^2\, d\xi=
\int\limits_{\mathbb S^1}\int\limits_0^\infty \frac{1}{\lambda}
|\mathfrak F g(\lambda\boldsymbol\omega)|^2 \lambda\, d\lambda
d\boldsymbol\omega= \\\int\limits_{\mathbb S^1}\int\limits_0^\infty
|\mathfrak F g(\lambda\boldsymbol\omega)|^2 \, d\lambda
d\boldsymbol\omega =\frac{1}{4\pi} \int_{\mathbb
S^1}\int_{-\infty}^\infty | \Phi(\boldsymbol\omega, \tau)|^2d\tau
d\boldsymbol\omega. \end{multline*} Recalling expression
\eqref{sobolev1} for $H^s$- norm we obtain the desired estimate
\eqref{sobolev6}.

\paragraph{Proof of Lemma \ref{sobolevlemma2}}
 Let $s>1/2$. It suffices to prove that
\begin{equation*}
\|uv\|_{H^{1/2}(\mathbb R^2)}\leq
 c\|u\|_{H^{1}(\mathbb R^2)} \|v\|_{H^{s}(\mathbb R^2)}
\end{equation*}
 for all $u\in H^{1}(\mathbb R^2)$ and for all $v\in H^{s}(\mathbb
R^2)$. Choose an arbitrary $u\in H^{1}(\mathbb R^2)$ and consider
the linear operator $\mathbf U:v\mapsto u\,v$. Set $\delta=s-1/2>0$.
Recall that $H^s(\mathbb R^2)$ coincides with $W^{s,2}(\mathbb
R^2)$. Since the embedding $H^{1+\delta}(\mathbb R^2)\hookrightarrow
L^\infty(\mathbb R^2)$ is bounded, we have
\begin{equation}\begin{split}\label{sobolev8ee}
    \|uv\|_{L^2(\mathbb R^2)}\leq \|v\|_{L^\infty(\mathbb R^2)}\,
    \|u\|_{L^2(\mathbb R^2)}\leq
    c\|v\|_{H^{1+\delta}(\mathbb R^2)}\|u\|_{H^{1}(\mathbb R^2)}
\end{split}\end{equation}
Since the embedding  $H^{1+\delta}(\mathbb R^2)\hookrightarrow W^{1,
2/(1-\delta)}(\mathbb R^2)$ and $H^{1}(\mathbb R^2)\hookrightarrow
L^{2/\delta}(\mathbb R^2)$ is bounded, see \cite{adams}, thm. 7.57,
we have
\begin{equation}\begin{split}\label{sobolev8eee}
    \|\nabla(uv)\|_{L^2(\mathbb R^2)}\leq
    \|v\|_{L^\infty(\mathbb R^2)}\,
    \|\nabla u\|_{L^2(\mathbb R^2)}+\|u\nabla
    v\|_{L^2(\Omega)}\leq\\
    c\|v\|_{H^{1+\delta}}\|u\|_{H^{1}(\mathbb R^2)}+
    \|\nabla v\|_{L^{ 2/(1-\delta)}(\mathbb
    R^2)}\|u\|_{L^{2/\delta}(\mathbb R^2)}\leq\\
    c\big( \| v\|_{H^{1+\delta}(\mathbb R^2)}+
    \| v\|_{W^{1, 2/(1-\delta)}(\mathbb R^2)}\,\big)
    \|u\|_{H^{1}(\mathbb R^2)}\leq c \| v\|_{H^{1+\delta}(\mathbb R^2)}
    \|u\|_{H^{1}(\mathbb R^2)}.
\end{split}\end{equation}
Combining \eqref{sobolev8ee}and \eqref{sobolev8eee} we obtain
\begin{equation}\label{sobolev8xx}
    \|\mathbf U v\|_{H^{1}(\mathbb R^2)}\leq c  \|u\|_{H^{1}(\mathbb R^2)}
\| v\|_{H^{1+\delta}(\mathbb R^2)}.
\end{equation}
On the other hand, the boundedness of the embedding
$H^{\delta}\hookrightarrow L^{ 2/(1-\delta)}(\mathbb R^2)$ implies
\begin{equation*}\begin{split}
    \|uv\|_{L^2(\mathbb R^2)}\leq
    \| v\|_{L^{ 2/(1-\delta)}(\mathbb
    R^2)}\|u\|_{L^{2/\delta}(\mathbb R^2)}\leq
    c \| v\|_{H^{\delta}(\mathbb R^2)}
    \|u\|_{H^{1}(\mathbb R^2)},
\end{split}\end{equation*}
which yields the estimate
\begin{equation*}
    \|\mathbf U v\|_{L^2(\mathbb R^2)}\leq c  \|u\|_{H^{1}(\mathbb R^2)}
\| v\|_{H^{\delta}(\mathbb R^2)}.
\end{equation*}
From this and \eqref{sobolev8xx} we conclude that   $\mathbf U$ is a
bounded operator from $H^{\delta}(\mathbb R^2)$ to $L^{ 2}(\mathbb
R^2)$ and from $H^{1+\delta}(\mathbb R^2)$ to $H^{1}(\mathbb R^2)$.
Moreover, its norm does not exceed $c\|u\|_{H^{1}(\mathbb R^2)}$.
Applying the interpolation theorem, \cite{bergh} Sec. 2.4, Sec. 6.4
Thm. 6.4.5, and noting that $1/2+\delta=s$ we obtain that the
desired inequality
\begin{equation*}
    \|uv\|_{H^{1/2}(\mathbb R^2)}\equiv \|\mathbf
    Uv\|_{H^{1/2}(\mathbb R^2)}\leq
   c\|u\|_{H^{1}(\mathbb R^2)} \|v\|_{H^{(\delta+1+\delta)/2}(\mathbb R^2)}=
 c\|u\|_{H^{1}(\mathbb R^2)} \|v\|_{H^{s}(\mathbb R^2)}
\end{equation*}
holds for all $u\in H^{1}(\mathbb R^2)$ and all $v\in H^{s}(\mathbb
R^2)$.

\end{appendix}

\end{document}